\documentclass[11pt]{amsart}

\usepackage[dvipsnames, table]{xcolor}
\usepackage{amsthm}
\theoremstyle{plain}
\usepackage{amssymb}
\usepackage{marvosym}
\usepackage{bm}
\usepackage{mathrsfs}
\usepackage{quiver}
\usepackage{enumerate}  
\usepackage{mathtools}
\usepackage{tikz-cd}
\usepackage{graphicx}
\usepackage{float}
\usepackage[titletoc,title]{appendix}
\usepackage{marginnote}
\usepackage[disable]{todonotes}
\usepackage{etoolbox}
\usepackage[all]{xy}
\makeatletter
\patchcmd{\Ginclude@eps}{"#1"}{#1}{}{}
\makeatother
\usepackage[outdir=./]{epstopdf}
\usepackage[numbers]{natbib}
\usepackage[utf8]{inputenc}
\usepackage[english]{babel}
\usepackage{changepage}
\usepackage{makecell}
\usepackage{enumitem}

\definecolor{lightblue}{HTML}{1F88CD}
\definecolor{lightgrey}{HTML}{727272}
\definecolor{lightblue2}{HTML}{009EC1}
\definecolor{mypink}{HTML}{FD00B0}
\definecolor{lightred}{HTML}{ff4d4d}

\usepackage{hyperref}
\hypersetup{
	colorlinks=true,
    linkcolor={OliveGreen},
    citecolor={blue},
	urlcolor={black}
}

\newtheorem*{theorem*}{Theorem}
\newtheorem{theorem}{Theorem}[section]

\newtheorem{lemma}[theorem]{Lemma}

\newtheorem{proposition}[theorem]{Proposition}
\theoremstyle{definition}

\theoremstyle{definition}
\newtheorem{definition}[theorem]{Definition}
\theoremstyle{definition}
\newtheorem{remark}[theorem]{Remark}
\theoremstyle{definition}

\theoremstyle{definition}

\theoremstyle{definition}

\theoremstyle{definition}

\theoremstyle{definition}

\theoremstyle{definition}
\newtheorem{question!}[theorem]{Question!}
\theoremstyle{definition}

\makeatletter
\newcommand*\sbt{\mathpalette\sbt@{.75}}
\newcommand*\sbt@[2]{\mathbin{\vcenter{\hbox{\scalebox{#2}{$\m@th#1\bullet$}}}}}
\makeatother

\newcommand{\sst}{\subset}

\let\emptyset\varnothing

\newcommand{\D}{\mathrm{D}}

\newcommand{\ZZ}{\mathbb{Z}}

\newcommand{\CC}{\mathbb{C}}

\newcommand{\ch}{\mathrm{ch}}

\newcommand{\pr}{\mathrm{pr}}

\newcommand{\RR}{\mathbb{R}}
\renewcommand\v{\mathbf v}
\newcommand\w{\mathbf w}

\DeclareMathOperator{\rk}{rk}

\DeclareMathOperator{\Coh}{\mathrm{Coh}}

\DeclareMathOperator{\Ext}{\mathsf{Ext}}

\DeclareMathOperator{\ext}{ext}

\DeclareMathOperator{\Gr}{Gr}

\newcommand{\cC}{\mathcal{C}}
\newcommand{\KK}{\mathrm{K}}
\newcommand{\cA}{\mathcal{A}}
\newcommand{\cE}{\mathcal{E}}

\newcommand{\cH}{\mathcal{H}}

\newcommand{\cK}{\mathcal{K}}
\newcommand{\cI}{\mathcal{I}}

\newcommand{\cT}{\mathcal{T}}
\newcommand{\cQ}{\mathcal{Q}}
\newcommand{\Ku}{\mathrm{Ku}}
\newcommand{\cP}{\mathcal{P}}
\newcommand{\cD}{\mathcal{D}}

\newcommand{\cM}{\mathcal{M}}

\DeclareMathOperator{\oh}{\mathcal{O}}
\DeclareMathOperator{\Hom}{\mathsf{Hom}}
\DeclareMathOperator{\RHom}{\mathsf{RHom}}

\newcommand{\zy}[1]{\textcolor{red}{#1}}

\usetikzlibrary{decorations.pathmorphing}

\title[A Note on Geometry of Bridgeland Moduli Spaces On Gushel-Mukai Threefolds]{A Note on Geometry of Bridgeland Moduli Spaces On Gushel-Mukai Threefolds}


\author{Songtao Ma}
\address{Department of Mathematics, University of Michigan - Ann Arbor, Ann Arbor, Michigan 48104, United States}
\email{skenma@umich.edu}

\date{\today}

\begin{document}

\begin{abstract}
Let $X$ be a Gushel--Mukai threefold and let $\operatorname{Ku}(X)$ be its Kuznetsov component. We show that the Bridgeland moduli spaces of semistable objects in $\operatorname{Ku}(X)$ are normal. Moreover, under certain numerical conditions, if the numerical class is primitive and $X$ is general, we prove that the corresponding Bridgeland moduli space is integral.

\end{abstract}

\maketitle

{
\hypersetup{linkcolor=blue}
\setcounter{tocdepth}{1}
}

\section{Introduction}

\subsection{Background and motivation}
Let $X$ be a smooth complex Gushel--Mukai threefold. It is a prime Fano threefold of Picard rank one and degree ten, and its bounded derived category admits a semiorthogonal decomposition
\begin{equation}\label{eq:intro-sod}
    \mathrm{D}^{\mathrm b}(X)=\left\langle \operatorname{Ku}(X),\mathcal O_X,\mathcal E^{\vee}\right\rangle,
\end{equation}
where $\mathcal E$ is the Mukai bundle and $\operatorname{Ku}(X)$ is the Kuznetsov component of $X$; see \cite{DK18,kuznetsov2018derived}. The category $\operatorname{Ku}(X)$ is an Enriques category: its Serre functor has the form $S_{\operatorname{Ku}(X)}\simeq \tau[2]$, where $\tau$ is a nontrivial involutive autoequivalence. Moreover, the equivariant category associated with the $\mathbb Z/2$-action generated by $\tau$ is a Calabi--Yau category of dimension two. This Calabi--Yau cover plays, in the categorical setting, the same role as the K3 cover of a classical Enriques surface \cite{kuznetsov2018derived,PPZ23}.

Bridgeland stability conditions were introduced in \cite{bridgeland}. Stability conditions on Kuznetsov components were constructed in \cite{bayer2017stability}, and the construction for Gushel--Mukai varieties was developed in \cite{PPZ22}. For a GM threefold, these stability conditions are Serre-invariant \cite{Per21}, and all Serre-invariant stability conditions belong to a single $\widetilde{\mathrm{GL}}_2^+(\mathbb R)$-orbit \cite{JLLZ,FeyzbakhshPertusi2021stab}. Consequently, they determine the same stable and semistable objects and hence the same moduli spaces.

Fix a Serre-invariant stability condition $\sigma$ and a numerical class $v\in \mathrm{K}_{\mathrm{num}}(\operatorname{Ku}(X))$. The stack of $\sigma$-semistable objects of class $v$ admits a proper good moduli space, denoted by $M_\sigma(\operatorname{Ku}(X),v)$; the corresponding relative construction is available in families \cite{bayer2021stability,PPZ23}. Perry--Pertusi--Zhao proved that these moduli spaces are nonempty for every numerical class and that, when $v$ is primitive, they are smooth and projective of the expected dimension for a generic GM threefold \cite{PPZ23}. The purpose of this paper is to study a global property not supplied by those results, namely irreducibility.

A difference between the moduli spaces of stable vector bundle on derived category of
Enriques surfaces, is the possible presence of singular stable points.
Indeed, an odd-square Mukai vector on an Enriques surface has odd rank,
and the determinant prevents a stable object $E$ from satisfying
$E\simeq E\otimes\omega_Y$; hence
$\operatorname{Ext}^2(E,E)=0$, and the stable moduli space is smooth
\cite[Remark~8.3]{Nue16b}. For $\operatorname{Ku}(X)$, has one no analogous determinant obstruction to
$E\simeq\tau E$. Thus moduli spaces of primitive odd-square classes
may contain obstructed stable objects, so one must control their
quadratic singularities before addressing normality. On the other hand, connectedness and irreducibility for the case of primitive character of Enriques surfaces, dealt with by \cite{Yos18}, depends on the geometry of Enriques surfaces. In particular, it based on \cite{GiesekerLi96} and \cite{OGrady96}, whose argument can not be directly duplicated in Gushel-Mukai threefold cases.

The numerical Grothendieck group $\mathrm{K}_{\mathrm{num}}(\operatorname{Ku}(X))$ has rank two, and in a suitable basis its Euler form is represented by $-I_2$ \cite{kuznetsov2018derived}. We write $v^2:=-\chi(v,v)$, so that the resulting quadratic form is positive definite. The parity of $v^2$ will enter the global connectedness argument, while lower bounds on $v^2$ control the codimension and rank estimates in the local normality argument.

\subsection{Main results}
Our principal result is the following.

\begin{theorem}\label{thm:intro-main}
Let $v$ be a fixed primitive numerical class with $v^2$ odd. Then, for a general smooth Gushel--Mukai threefold $X$, the moduli space $M_\sigma(\operatorname{Ku}(X),v)$ is irreducible.
\end{theorem}

Here and below, for a fixed class $v$, a general GM threefold means a point of a nonempty Zariski-open subset of the relevant moduli space of GM threefolds; this open subset is allowed to depend on $v$. Since all Serre-invariant stability conditions lie in the same $\widetilde{\mathrm{GL}}_2^+(\mathbb R)$-orbit, the conclusion is independent of the choice of $\sigma$.

Up to sign, there are two primitive numerical classes of square one. For a general GM threefold, their Bridgeland moduli spaces are identified in \cite{JLLZ} with the minimal model of the Fano surface of conics and with a Gieseker moduli space of rank-two sheaves, respectively. In particular, the square-one case of Theorem~\ref{thm:intro-main} is already known. Thus the new argument in this paper concerns primitive classes of odd square at least five. Its local input is the following normality theorem, which is also useful for nonprimitive classes.

\begin{theorem}\label{thm:intro-normality}
Let $X$ be a smooth Gushel--Mukai threefold, set $\mathcal C=\operatorname{Ku}(X)$, and let $\sigma$ be a Serre-invariant stability condition on $\mathcal C$. If $v_0\in \mathrm{K}_{\mathrm{num}}(\mathcal C)$ is primitive and $v_0^2\geq 5$, then $M_\sigma(\mathcal C,mv_0)$ is normal for every integer $m\geq 1$.
\end{theorem}

The importance of Theorem~\ref{thm:intro-normality} is that $M_\sigma(\mathcal C,mv_0)$ may contain strictly semistable objects and may also be singular along the stable locus. Its normality therefore cannot be deduced merely from the existence of a dense smooth locus.

\subsection{Strategy of the proof}

For the new cases $v^2\geq 5$, the proof has three steps. First, we prove Theorem~\ref{thm:intro-normality}. At a polystable object
$E\simeq \bigoplus_i E_i\otimes V_i$ of class $mv_0$, injectivity of the central charge gives $[E_i]=m_i v_0$ for every stable factor. Formality identifies the analytic germ at $[E]$ with a reductive quotient
$Z_{\mathbf d}\mathbin{/\!\!/}G_{\mathbf d}$, where $Z_{\mathbf d}$ is the zero fibre of a Serre-twisted quadratic Kuranishi map \cite{ChenPertusiZhao2024}. Induction on $|\mathbf d|$ shows that the non-simple locus has codimension at least two. At a simple point, an étale comparison with stable objects gives either a smooth local model or, at a $\tau$-fixed object, a quadratic hypersurface of rank at least three. Hence $Z_{\mathbf d}$ is a reduced local complete intersection satisfying $R_1$ and $S_2$, and its reductive quotient is normal.

Second, let $M_\sigma(\mathcal C,v)$ be singular and normal, with $v$ primitive and $v^2$ odd. Assuming that it is disconnected, let $Y$ be the component containing a singular point $p$. Adapting the blow-up argument of Kaledin--Lehn--Sorger \cite{kaledin06Sin}, elementary modifications of a relative Ext complex on $\operatorname{Bl}_p(Y)$ give a Chern-class identity for the exceptional divisor $D$. Since $d=\dim Y=v^2+1$ is even, its top-degree part gives $D^d=0$, whereas
$\deg(D^d)=(-1)^{d-1}\operatorname{mult}_p(Y)\neq 0$. This contradiction proves connectedness, and normality then implies irreducibility.

Finally, for the fixed class $v$, we choose a special GM threefold whose moduli space contains a $\tau$-fixed stable object and is therefore singular. The first two steps make this special fibre irreducible. We place it in a one-parameter family and use the proper relative Bridgeland moduli space provided by stability conditions in families \cite{bayer2021stability}. Normality and equidimensionality of the fibres then imply that irreducibility persists over a nonempty open subset of the base, proving Theorem~\ref{thm:intro-main}.

\subsection{Related work}
The geometry of moduli spaces on Enriques categories is modeled in part on the classical theory of sheaves on Enriques surfaces. Stable vector bundles and their relation to the K3 cover were studied by Kim \cite{Kim98,Kim06}. Semistable sheaves and reducibility phenomena were investigated by Hauzer \cite{Hau10}, while existence results for stable sheaves were obtained by Nuer and Yoshioka \cite{Nue16a,Yos17,Yos18}. Bridgeland moduli spaces on Enriques surfaces, including their projectivity, wall-crossing, and birational geometry, were studied in \cite{Nue16b,Bec20,NY20}.

For K3 surfaces, Mukai constructed the natural symplectic form on moduli spaces of stable sheaves \cite{Muk84}, and Yoshioka proved fundamental irreducibility results \cite{Yos99}. Bridgeland moduli spaces and their wall-crossing birational geometry were developed systematically by Bayer and Macrì \cite{BM14a,BM14b}. Singular moduli spaces associated with nonprimitive classes were studied through quiver and symplectic methods by Kaledin--Lehn--Sorger and Arbarello--Saccà \cite{kaledin06Sin,ArbarelloSacca2018}; related results for Kuznetsov components appear in \cite{Sac23}. These works provide both a model and a contrast for the present setting: stable moduli spaces in a two-Calabi--Yau category are smooth, whereas the fixed locus of the Serre involution can produce singular stable points in an Enriques category.

Stability conditions on Kuznetsov components were constructed in \cite{bayer2017stability}, and the Gushel--Mukai case was developed in \cite{PPZ22}. Serre-invariance for GM threefolds was proved in \cite{Per21}. Perry--Pertusi--Zhao established a general theory of moduli spaces in Enriques categories and related them to fixed loci in moduli spaces on their Calabi--Yau covers \cite{PPZ23}. Further results on moduli spaces in Kuznetsov components of Fano threefolds and GM varieties can be found in \cite{APR22,JLLZ,FGLZ25,FGLZcubic}. The local normality argument of this paper also relies on the formality results of \cite{ChenPertusiZhao2024} and is influenced by the quiver-theoretic methods of \cite{kaledin06Sin,ArbarelloSacca2018}.

\subsection{Organization of the paper}
In Section~2, we recall semiorthogonal decompositions, Kuznetsov components of GM threefolds, their numerical Grothendieck lattices, and the Enriques structure together with its Calabi--Yau cover. In Section~3, we review Serre-invariant stability conditions, the induced stability condition on the cover, and the construction of absolute and relative moduli spaces.

Section~4 is devoted to normality. We construct the Serre-twisted quadratic local model, compare its simple locus with stable objects, establish the rank estimate at Serre-fixed stable points, and prove Theorem~\ref{thm:intro-normality}. In Section~5, we adapt the blow-up argument of Kaledin--Lehn--Sorger to prove the connectedness criterion for singular normal moduli spaces of primitive odd-square classes.

In Section~6, we construct a singular irreducible fibre using a special GM threefold and its Calabi--Yau cover, and then deform this fibre in a family to prove Theorem~\ref{thm:intro-main}.

\subsection{Acknowledgements}
This paper grew out of my undergraduate thesis. I am deeply grateful to my advisor at that time, Andrei C\u{a}ld\u{a}raru, for his generous guidance. Most of this paper was finished during my graduate student stage. I sincerely thank my advisor, Alexander Perry, for his helpful comments and advice. The author also thank Ritwick Bhargava, Zhiyu Liu, and Saket Shah for useful discussions. Finally, I am grateful to my parents for their constant support.

\subsection*{Notation} \leavevmode\\
    \indent $\D^b(X)$: the bounded derived category of coherent sheaves on $X$\\
    \indent $\Ku(X)$: the Kuznetsov component of $\D^b(X)$\\
    \indent $\ch(E)$: the Chern character of an object $E\in \D^b(X)$\\
     \indent $c_i(E)$: the $i$-th Chern class of an object $E \in \D^b(X)$\\
    \indent $\cH^i(E)$: the $i$-th cohomology object of a complex $E$ with respect to the heart $\Coh(X)$.  \\
    \indent $H^i(E)$: the $i$-th cohomology of a sheaf $E$\\
\subsection*{Conventions} \leavevmode
\begin{itemize}
    \item Let $\sigma$ be a weak stability condition. Then the central charge and heart are denoted by $Z_{\sigma}$ and $\cA_{\sigma}$,  respectively.

    \item We use $\hom$ and $\ext^{i}$ to represent the dimension of the vector spaces $\Hom$ and~$\Ext^{i}$.

\item For a lattice $\Lambda$, we set $\Lambda_{\RR}\coloneqq \Lambda\otimes_{\ZZ} \RR$.
    
    \item We denote the numerical class in the numerical Grothendieck group by $[E]$ for any object $E$. In our setting, giving a numerical class is equivalent to giving a Chern character.
\end{itemize}

\section{Kuznetsov components} \label{Kuz_component}

In this section, we collect some definitions and properties of Kuznetsov components.

\subsection{Semiorthogonal decompositions} 

We begin with a series of general definitions.

\begin{definition}
Let $\cT$ be a triangulated category and $E \in \cT$ be an object. We say that $E$ is an \emph{exceptional object} if $\RHom(E, E) = k$. Now let $(E_1, \dots, E_m)$ be a collection of exceptional objects in $\cT$. We say it is an \emph{exceptional collection} if $\RHom(E_i, E_j) = 0$ for any $i > j$.
\end{definition}

\begin{definition}
Let $\cT$ be a triangulated category and $\cA$ a triangulated subcategory. We define the \emph{right orthogonal complement} of $\cA$ in $\cT$ as the full triangulated subcategory
\[ \cA^\bot := \{ X \in \cT \mid \RHom(Y, X) =0  \text{ for all } Y \in \cA \}.  \]
The \emph{left orthogonal complement} is defined similarly, as 
\[ {}^\bot \cA := \{ X \in \cT \mid \RHom(X, Y) =0  \text{ for all } Y \in \cA \}.  \]
\end{definition}

\begin{definition}
Let $\mathcal{B}$ be a triangulated category, then a triangulated subcategory $\mathcal A\subset \mathcal B$ is \textit{admissible} if the inclusion functor $i\colon \mathcal{A} \hookrightarrow \mathcal{B}$ has left adjoint $i^*$ and right adjoint $i^!$.
\end{definition}

\begin{definition}
A sequence of admissible subcategories $( \cA_1, \dots, \cA_m  )$ in $\cT$ is \emph{semiorthogonal} if $\cA_j \sst \cA_i^\bot $ for all $i > j$.
The triangulated subcategory generated by $\cA_1,\cdots,\cA_m$ is denoted by $ \langle \cA_1, \dots, \cA_m \rangle$.

A semiorthogonal collection $( \cA_1, \dots, \cA_m  )$ of $\cT$  is a \emph{semiorthogonal decomposition} if $ \cT = \langle \cA_1, \dots, \cA_m \rangle$.
\end{definition}

\begin{definition} Let $\cA \subset \cT$ be an admissible triangulated subcategory, then the \textit{left mutation functor} through $\cA$ is the functor $\mathbb{L}_\mathcal{A}$ defined by the canonical functorial exact triangle 
\[ii^! \rightarrow id \rightarrow \mathbb{L}_{\cA}.\]
\end{definition}

In particular, when $E\in \cT$ is an exceptional object and $F\in \cT$ is any object, $\mathbb{L}_E F$ is defined by the following exact triangle
\begin{equation} \label{mutation}
    E\otimes \RHom (E,F) \rightarrow F \rightarrow \mathbb{L}_E F.
\end{equation}
Therefore, it is clear that $\RHom(F,E) = 0$ is equivalent to  $\mathbb{L}_E F = F$.

\subsection{Kuznetsov components of Gushel--Mukai threefolds}
Let $X$ be a Gushel--Mukai (GM) threefold, which is a prime Fano threefold of degree $10$. It is either ordinary, in which case it is a quadric section of a linear section of $\operatorname{Gr}(2,5)$, or special, in
which case it is a double cover of a linear section of $\operatorname{Gr}(2,5)$. There exists a unique stable rank-two vector bundle $\mathcal E$, called the Mukai bundle,
with
\[\operatorname{ch}(\mathcal E)=2-H+L+\frac13 P,\]
which is the pullback of the tautological bundle on $\Gr(2,5)$. Here, $H$, $L$, and $P$ are the classes of hyperlane sections, lines, and points of $X$, respectively.

By \cite{kuznetsov2009derived}, we have the following semiorthogonal decomposition 
\[\D^b(X) = \langle \Ku(X), \oh_X , \cE^\vee \rangle.\]
We denote the left adjoint of $\Ku(X)\hookrightarrow \D^b(X)$ by $\pr$, which is called the \emph{projection functor}. It can be written as a composition of left mutations  $\pr=\mathbb{L}_{\oh_X}\mathbb{L}_{\cE^{\vee}}$.

Let $\mathrm{K}_0(\Ku(X))$ be the Grothendieck group of $\Ku(X)$. The \textit{numerical Grothendieck group} is the quotient group $\mathrm{K}_{\mathrm{num}}(\Ku(X)) = \mathrm{K}_0(\Ku(X))/\ker(\chi)$, where $\chi$ is the Euler form. 

For Kuznetsov components of Gushel--Mukai threefolds, we have the following computational result.

\begin{lemma} [{\cite{kuznetsov2009derived,kuznetsov2018derived}}] \label{NumKGroup}
The numerical Grothendieck group $\mathrm{K}_{\mathrm{num}}(\Ku(X))$ is a rank 2 lattice with basis vectors
\[\v=1-2L, \quad \w=2-H+\frac{5}{6}P.\]

The Euler form with respect to the basis is 
\[\left(\begin{array}{cc}-1 & 0 \\0 & -1 \\\end{array}\right).\]
\end{lemma}

A numerical class $\mathbf{c}\in \mathrm{K}_{\mathrm{num}}(\Ku(X))$ is a \emph{$(-r)$-class} if $\chi(\mathbf{c}, \mathbf{c})=-r$. It is clear that the only $(-1)$-classes in $\mathrm{K}_{\mathrm{num}}(\Ku(X))$ are $\v$ and $\w$ up to sign.

One can easily verify that $$\chi(a+bH+cL+dP) = a + \frac{17}{6}b + \frac{1}{2}c + d.$$

\subsection{The Enriques structure and the Calabi--Yau cover}
\label{subsec:enriques-cover}

The Kuznetsov component $\Ku(X)$ of a GM threefold $X$ is an Enriques category. More precisely, the following result combines \cite[Propositions~2.6--2.7]{kuznetsov2018derived} with \cite[Lemma~7.3]{PPZ23}.

\begin{proposition}
\label{prop:enriques-structure}
There is a nontrivial involutive autoequivalence $\tau\colon\Ku(X)\to\Ku(X)$, underlying a $\mathbb Z/2$-action, such that
$S_{\Ku(X)}\simeq\tau[2]$. Moreover, $\tau$ acts trivially on $\mathrm{K}_{\mathrm{num}}(\Ku(X))$.
\end{proposition}

Let
$\Ku(X)^{\mathbb Z/2}$
be the corresponding equivariant category. By \cite[Lemma~3.4 and Example~3.10]{PPZ23}, the category $\Ku(X)^{\mathbb Z/2}$ has the Serre functor $S_{\Ku(X)^{\mathbb Z/2}}\simeq [2]$, called the Calabi--Yau cover of $\Ku(X)$, and it carries a residual $\mathbb Z/2$-action whose generator will be denoted by $\iota$. Taking invariants for the residual action recovers $\Ku(X)$. Furthermore, $\Ku(X)^{\mathbb Z/2}$ is equivalent to $\operatorname{Ku}(X^{\mathrm{op}})$ for an opposite GM variety $X^{\mathrm{op}}$ which is a K3 surface of degree $10$ when $X$ is special, and is a GM fourfold when $X$ is ordinary.

We denote by
$\operatorname{Forg}\colon\Ku(X)^{\mathbb Z/2}\to\Ku(X)$
the forgetful functor and by
$\operatorname{Inf}\colon\Ku(X)\to\Ku(X)^{\mathbb Z/2}$
the inflation functor. These functors are mutually left and right adjoint. For the group $\mathbb Z/2$, the reconstruction formalism of \cite[Section~2.3]{PPZ23} gives
\begin{equation}
\operatorname{Forg}\circ\operatorname{Inf}
\simeq
\operatorname{id}_{\Ku(X)}\oplus\tau,
\qquad
\operatorname{Inf}\circ\operatorname{Forg}
\simeq
\operatorname{id}_{\Ku(X)^{\mathbb Z/2}}\oplus\iota.
\label{eq:inf-forg}
\end{equation}
The residual involution is obtained by changing the sign of the equivariant structure. In particular, $\iota\circ\operatorname{Inf}\simeq\operatorname{Inf}$.


\begin{lemma}
\label{lem:inflation-pairing}
For any $a,b\in \mathrm{K}_{\mathrm{num}}(\Ku(X))$, one has
\begin{equation}
\chi\bigl(\operatorname{Inf}_*a,\operatorname{Inf}_*b\bigr)
=
2\chi(a,b).
\label{eq:inflation-pairing}
\end{equation}
Consequently, $\operatorname{Inf}_*$ is injective and
$\Pi_{\Ku(X)}:=\operatorname{Inf}_*
\bigl(\mathrm{K}_{\mathrm{num}}(\Ku(X))_{\mathbb R}\bigr)$
is a positive-definite two-plane in $\mathrm{K}_{\mathrm{num}}(\Ku(X)^{\mathbb Z/2})_{\mathbb R}$. The residual involution $\iota$ fixes $\Pi_{\Ku(X)}$ pointwise.
\end{lemma}

\begin{proof}
By adjunction and \eqref{eq:inf-forg}, we have
\[
\chi\bigl(\operatorname{Inf}_*a,\operatorname{Inf}_*b\bigr)=
\chi\bigl(a,b+\tau_*b\bigr)
=2\chi(a,b),
\]
where the last equality follows from Proposition~\ref{prop:enriques-structure}. Positive definiteness of the Euler form of $\Ku(X)$ then implies that $\operatorname{Inf}_*$ is injective and that its image is a positive two-plane. Finally, $\iota\circ\operatorname{Inf}\simeq\operatorname{Inf}$, so $\iota$ fixes this image.
\end{proof}


\section{Stability conditions}
\label{sec:stability}

In this section, we fix $X$ to be a GM threefold and review the construction and properties of stability conditions on $\Ku(X)$.

\subsection{Stability conditions and moduli spaces}
\label{subsec:numerical-stability}

\begin{definition}
A \emph{slicing} $\cP$ of a triangulated category $\cT$ consists of full additive subcategories $\cP(\phi) \subset \cT$ for each $\phi \in \mathbb{R}$ satisfying
\begin{enumerate}[resume]
\item for $\phi\in (0,1]$, the subcategory $\cP(\phi)$ is given by the zero object and all $\sigma$-semistable objects whose phase is $\phi$;
\item for $\phi + n$ with $\phi\in (0,1]$ and $n\in \mathbb{Z}$, we set $\cP(\phi + n) := \cP(\phi)[n]$.
\end{enumerate}
\end{definition}

The nonzero objects of $\cP(\phi)$ are called \emph{semistable} of \emph{phase} $\phi$; the simple objects in $\cP(\phi)$ are called \emph{stable}. 
For a nonzero object $F\in\cD$, the sequence of morphisms appearing in the above definition is unique and called the \emph{Harder--Narasimhan (HN) filtration} of $F$, and the objects $A_i$ are called the \emph{Harder--Narasimhan factors} of $F$.
       
Let $\mathcal T$ be a triangulated category and fix a surjective homomorphism $v\colon \KK_0(\cT)\to \Lambda$ to a finite-rank lattice. 

\begin{definition}
\label{definition-stability-condition}
Consider a pair $\sigma=(Z,\cP)$, where $Z\colon\Lambda\rightarrow\CC$ is a homomorphism and $\cP$ is a slicing of $\cT$. 
We say that $\sigma$ is a \emph{stability condition} on $\cT$ with respect to $(\Lambda,v)$ if
\begin{enumerate}
    \item for all phases $\phi \in \RR$ and objects $0 \neq F \in \cP(\phi)$, we have $Z(v(F)) \in\RR_{>0}\cdot e^{i\pi\phi}$, and
    \item there exists a quadratic form $Q$ on the vector space $\Lambda_\RR\coloneqq\Lambda\otimes\RR$ such that
\begin{itemize}
\item $Q$ is negative definite on the kernel of the $\RR$-linear extension $Z_{\RR}\colon\Lambda_\RR\to\CC$, and
\item for any $\phi\in \RR$ and any $F \in \cP(\phi)$, we have $Q(v(F)) \geqslant 0$. 
\end{itemize}
\end{enumerate}
\end{definition}

An object $F \in \cP(\phi)$ is called a \emph{$\sigma$-semistable object of phase $\phi$}.

In our paper, we will always assume that $\cT$ is a semiorthogonal component of a smooth projective variety, and take $v$ to be the natural quotient map $\KK_0(\cT)\twoheadrightarrow \KK_{\mathrm{num}}(\cT)$.

\subsection{Serre-invariant stability conditions}
\label{subsec:serre-invariant-stability}

The universal covering $\widetilde{\mathrm{GL}}_2^+(\mathbb R)$ of $\mathrm{GL}_2^+(\mathbb R)$ acts on the right
on the set of stability conditions, while exact autoequivalences of
$\mathcal T$ act on the left (cf.~\cite{bridgeland}).

\begin{definition}
\label{def:invariant-stability}
Let $\Phi$ be an autoequivalence of $\mathcal T$. A stability
condition $\sigma$ is called $\Phi$-invariant if
$\Phi\cdot\sigma=\sigma\cdot\widetilde g$ for some
$\widetilde g\in\widetilde{\mathrm{GL}}_2^+(\mathbb R)$. If
$\mathcal T$ has a Serre functor $S_{\mathcal T}$, we say that $\sigma$ is \emph{Serre-invariant} when it is $S_{\mathcal T}$-invariant.
\end{definition}

By \cite{bayer2017stability,JLLZ,FeyzbakhshPertusi2021stab,Per21}, there exists a Serre-invariant stability condition $\sigma$ on $\Ku(X)$ for any GM threefold $X$ and all Serre-invariant stability conditions on $\Ku(X)$ belong to the same $\widetilde{\mathrm{GL}}_2^+(\mathbb R)$-orbit. Note that in this case, we have $\tau\cdot\sigma=\sigma$.

\begin{remark}\label{rmk:injective-Z}
If $\sigma=(Z, \cP)$ is a Serre-invariant stability condition on $\Ku(X)$, then by \cite[Lemma~A.8]{JLLZ}, $Z\colon \KK_{\mathrm{num}}(\Ku(X))\to \CC$ is injective. Hence, $Z_{\RR}$ is an isomorphism between $\RR$-vector spaces.
\end{remark}

Any Serre-invariant stability condition on $\Ku(X)$ induces a stability condition on $\Ku(X)^{\ZZ/2}$. More precisely, we have:

\begin{theorem}
\label{thm:cover-stability}
Let $X$ be a GM threefold and $\sigma=(Z,\cP)$ be a Serre-invariant stability condition on $\Ku(X)$. There is a stability condition
$\omega=(Z_{\omega},\cP_{\omega})$ on
$\Ku(X)^{\ZZ/2}$, where
\[
\cP_{\omega}(\phi)
=
\left\{
F\in \Ku(X)^{\ZZ/2}\mid\operatorname{Forg}(F)\in\cP(\phi)
\right\},
\qquad
Z_{\omega}(F)
=
Z\bigl(\operatorname{Forg}(F)\bigr).
\]
For any $0\neq F\in\mathcal \Ku(X)^{\ZZ/2}$, the following statements hold:
\begin{enumerate}
\item $F$ is $\omega$-semistable if and only if
$\operatorname{Forg}(F)$ is $\sigma$-semistable;
\item if $F$ is $\omega$-stable, then
$\operatorname{Forg}(F)$ is $\sigma$-polystable;
\item if $\operatorname{Forg}(F)$ is $\sigma$-stable, then $F$
is $\omega$-stable.
\end{enumerate}
Moreover, the residual involution $\iota$ fixes
$\omega$.
\end{theorem}

\begin{proof}
The construction and statements are
\cite[Theorem~4.8]{PPZ23}. The final assertion follows directly from
the construction, since
$\operatorname{Forg}\circ\iota\simeq\operatorname{Forg}$.
\end{proof}

We end this section by discussing moduli spaces of semistable objects.

Let $X$ be a GM threefold and $0\neq v\in \KK_{\rm num}(\Ku(X))$. For any Serre-invariant stability condition $\sigma$ on $\Ku(X)$, we denote by $\cM_{\sigma}(v,\phi)$ be the moduli stack of $\sigma$-semistable objects with phase $\phi$ and numerical class $v$. The substack of $\sigma$-stable objects is denoted by $\mathcal M_\sigma^{\mathrm{st}}(v,\phi)$. By \cite[Theorem~21.24]{BLMNSP21}, $\cM_{\sigma}(v,\phi)$ is an Artin stack of finite type over $\CC$ and $\mathcal M_\sigma^{\mathrm{st}}(v,\phi)$ is an open substack of $\cM_{\sigma}(v,\phi)$. Moreover, $\cM_{\sigma}(v,\phi)$ admits a proper good moduli space $M_{\sigma}(v,\phi)$ whose closed points correspond to
S-equivalence classes of $\sigma$-semistable objects, or equivalently to
isomorphism classes of $\sigma$-polystable objects.

\section{Normality of bridgeland moduli space}

Throughout this section, we fix $X$ to be a GM threefold, $\cC\coloneqq \Ku(X)$.

\subsection{Local Model}

We will use the following result from \cite{ChenPertusiZhao2024} as the tool for calculating the local model of the moduli space $\bigl(M_\sigma(C,v),[F]\bigr)$.

\begin{theorem}
\label{thm:CPZ-kuranishi}
Let $X$ be a smooth Gushel--Mukai threefold, set
$\mathcal C=\operatorname{Ku}(X)$, and let $\sigma$ be a
Serre-invariant stability condition on $\mathcal C$. Let
$F\in\mathcal C$ be a $\sigma$-polystable object of numerical
class $v$, and set
$G_F:=\operatorname{Aut}_{\mathcal C}(F)$. Then the derived
endomorphism DG-Lie algebra
$\mathbf R\Hom_{\mathcal C}(F,F)$ is formal.

Let
\[
\kappa_{2,F}\colon
\operatorname{Ext}^1_{\mathcal C}(F,F)
\longrightarrow
\operatorname{Ext}^2_{\mathcal C}(F,F),
\qquad
x\longmapsto x\circ x,
\]
be the quadratic part of the Kuranishi map. Then there is an
isomorphism
\[
\bigl(M_\sigma(\mathcal C,v),[F]\bigr)
\simeq
\bigl(
\kappa_{2,F}^{-1}(0)
\mathbin{/\!\!/}G_F,0
\bigr)
\]
of germs of complex analytic spaces, where $G_F$ acts naturally on
$\operatorname{Ext}^1_{\mathcal C}(F,F)$.
\end{theorem}

\begin{proof}
The formality assertion follows from
\cite[Corollary~3.12]{ChenPertusiZhao2024}. Formality implies that the
Kuranishi space satisfies the quadracity property, and the asserted
description of the analytic germ of the good moduli space is
\cite[Proposition~3.11]{ChenPertusiZhao2024}.
\end{proof}

\subsection{The Serre-twisted quadratic model}

Fix a closed point of $M_\sigma(C,mv_0)$, represented by a $\sigma$-polystable
object $E\simeq\bigoplus_{i\in I}E_i\otimes V_i$, where the $E_i$ are pairwise
non-isomorphic stable objects of the same phase. By Remark~\ref{rmk:injective-Z},
after deleting zero summands we have $[E_i]=m_i v_0$ with $m_i>0$.

Notice $\tau E_i$ is again stable of the same phase. We enlarge $I$, if necessary, by adding the $\tau$-translates of the stable factors
and allowing the corresponding multiplicity spaces $V_i$ to be zero. Thus we may
assume that there is an involution $\nu\colon I\to I$ such that $\tau(E_i)\simeq
E_{\nu(i)}$.

By \cite[Lemma~4.3(i)]{FGLZ25}, for an odd-dimensional Gushel--Mukai variety the
involution $S_{\Ku(X)}[-2]$ acts trivially on $K_{\mathrm{num}}(\Ku(X))$. Since
$X$ is a GM threefold and $\tau=S_C[-2]$, the involution $\tau$ acts trivially
on $K_{\mathrm{num}}(C)$. Hence
$[E_{\nu(i)}]=[\tau(E_i)]=[E_i]=m_i v_0$,
and thus $m_{\nu(i)}=m_i$.

For $i,j\in I$, set
$A_{ij}:=\operatorname{Ext}^1_C(E_i,E_j)$ and $a_{ij}:=\dim A_{ij}$. For a
dimension vector $\mathbf d=(d_i)_{i\in I}$, choose vector spaces $U_i$ with
$\dim U_i=d_i$. We use the following notation throughout the rest of the section:
\begin{enumerate}
    \item $\text{local representation  space}: \operatorname{Rep}_{\mathbf d}
:=\bigoplus_{i,j\in I}\Hom(U_i,U_j)\otimes A_{ij}$
    \item $\text{obstruction space}: \operatorname{Obs}_{\mathbf d}
:=\bigoplus_{i\in I}\Hom(U_i,U_{\nu(i)})$
    \item $\text{stablizer group}: G_{\mathbf d}:=\prod_{i\in I}\operatorname{GL}(U_i)$
    \item $\text{expected dimesion of zero fibre}: \operatorname{expdim}(\mathbf d):=\dim\operatorname{Rep}_{\mathbf d}
-\dim\operatorname{Obs}_{\mathbf d}$
    \item\text{weight}: $\operatorname{wt}(\mathbf d):=\sum_{i\in I}m_i d_i$
\end{enumerate}
The Serre pairings $A_{ij}\otimes A_{j,\nu(i)}\to\mathbb C$ define a quadratic map
$\mu_{\mathbf d}\colon \operatorname{Rep}_{\mathbf d}\longrightarrow
\operatorname{Obs}_{\mathbf d}.$ After choosing Serre-dual bases $B_{ij}\subset A_{ij}$ and
$B_{ij}^\vee\subset A_{j,\nu(i)}$, this map may be written as
\begin{equation*}
(\mu_{\mathbf d}(x))_i=
\sum_{j\in I}\sum_{\alpha\in B_{ij}}\epsilon_\alpha x_{\alpha^\vee}x_\alpha
\in \Hom(U_i,U_{\nu(i)}),
\end{equation*}
where $\epsilon_\alpha\in\{\pm 1\}$ depends only on the chosen sign convention for
the Yoneda product. Finally, define the zero fibre to be:
\begin{equation*}
\mathfrak Z_{\mathbf d}:=\mu_{\mathbf d}^{-1}(0)\subset \operatorname{Rep}_{\mathbf d}.
\end{equation*}

\begin{lemma}\label{lem:serre-twisted-model}
Let $\mathbf d=(\dim V_i)_{i\in I}$. Then
\begin{equation*}
\operatorname{Ext}^1_C(E,E)\simeq \operatorname{Rep}_{\mathbf d},
\qquad
\operatorname{Ext}^2_C(E,E)\simeq \operatorname{Obs}_{\mathbf d}.
\end{equation*}
Under these identifications, $\kappa_{2,E}=\mu_{\mathbf d}$. Consequently,
\begin{equation*}
\bigl(M_\sigma(C,mv_0),[E]\bigr)\simeq
\bigl(\mathfrak Z_{\mathbf d}//G_{\mathbf d},0\bigr).
\end{equation*}
Moreover, $a_{ij}=v_0^2m_i m_j+\delta_{ij}+\delta_{j,\nu(i)}.$

\end{lemma}

\begin{proof}
The decomposition of $E$ gives
\begin{equation*}
\operatorname{Ext}^1_C(E,E)\simeq
\bigoplus_{i,j}\Hom(U_i,U_j)\otimes\operatorname{Ext}^1_C(E_i,E_j)
=\operatorname{Rep}_{\mathbf d}.
\end{equation*}
Since $S_C\simeq \tau[2]$, Serre duality gives
$A_{ij}^\vee\simeq \operatorname{Ext}^1_C(E_j,\tau E_i)\simeq A_{j,\nu(i)}$. Similarly,
\begin{equation*}
\operatorname{Ext}^2_C(E_i,E_j)\simeq
\Hom_C(E_j,\tau E_i)^\vee\simeq
\Hom_C(E_j,E_{\nu(i)})^\vee.
\end{equation*}
Since $E_j$ and $E_{\nu(i)}$ are stable objects
of the same phase, we have
\begin{equation*}
\operatorname{Ext}^2_C(E_i,E_j)\simeq
\begin{cases}
\mathbb C, & j=\nu(i),\\
0, & j\ne \nu(i).
\end{cases}
\end{equation*}
Thus $\operatorname{Ext}^2_C(E,E)\simeq\operatorname{Obs}_{\mathbf d}$. Under these
identifications, the quadratic part of the Kuranishi map is the Yoneda product
$x\mapsto x\circ x$, which is exactly the Serre-twisted map $\mu_{\mathbf d}$.
The analytic local model then follows from Theorem~\ref{thm:CPZ-kuranishi}.

It remains to compute $a_{ij}$. Since $[E_i]=m_i v_0$ and $[E_j]=m_jv_0$, we
have $-\chi(E_i,E_j)=v_0^2m_im_j$. Moreover,
$\Hom(E_i,E_j)=\delta_{ij}$ and, by the calculation above,
$\operatorname{ext}^2(E_i,E_j)=\delta_{j,\nu(i)}$. Therefore
$a_{ij}=\operatorname{ext}^1(E_i,E_j)=v_0^2m_im_j+\delta_{ij}+\delta_{j,\nu(i)}$.
\end{proof}

A point $x\in \mathfrak Z_{\mathbf d}$ is a collection of linear maps
$x_\alpha\colon U_i\to U_j$, indexed by
$\alpha\in B_{ij}$, and a subrepresentation of $x$ is a collection of
subspaces $U_i'\subset U_i$ such that
$x_\alpha(U_i')\subset U_j'$ for every
$\alpha\in B_{ij}$. We call $x$ \emph{simple} if its only
subrepresentations are the zero collection and the whole collection
$(U_i)_{i\in I}$. Equivalently, $x$ defines a simple module over the
quadratic algebra determined by the relations $\mu_{\mathbf d}=0$.
In this case, Schur's lemma identifies
$\operatorname{Stab}_{G_{\mathbf d}}(x)$ with the diagonal scalar
subgroup $\mathbb G_m$, and the $G_{\mathbf d}$-orbit of $x$ is closed. The goal of the following lemma is to identify simple points in $\mathfrak Z_{\mathbf{d}}$ with $\sigma$-stable objects in $C$.

We denote $E_{\mathbf{d}} := \oplus_i E_i \otimes U_i$ respect to the dimension vector $\mathbf{d}$.

\begin{lemma}
\label{lem:etale-chart-simple-points}
Let $\mathbf d$ be a dimension vector. The pointed stacks
\[
\bigl([\mathfrak Z_{\mathbf d}//G_{\mathbf d}],[0]\bigr)
\quad\text{and}\quad
\Bigl(
\mathcal M_\sigma
\bigl(\mathcal C,\operatorname{wt}(\mathbf d)v_0\bigr),
[E_{\mathbf d}]
\Bigr)
\]
admit a common affine étale neighborhood. The two maps may be chosen
to preserve the stabilizer at the distinguished point and, after
shrinking, to be cartesian over étale maps of the corresponding good
moduli spaces. Consequently, if $x\in \mathfrak Z_{\mathbf d}$ is simple, then, after replacing
$x$ by a nonzero scalar multiple, it determines a $\sigma$-stable
object $S_x$ of class $\operatorname{wt}(\mathbf d)v_0$.

\end{lemma}

\begin{proof}
The stabilizer of $[0]\in[\mathfrak Z_{\mathbf d}//G_{\mathbf d}]$ is
$G_{\mathbf d}$, and
\[
\operatorname{Aut}_{\mathcal C}(E_{\mathbf d})
\simeq
G_{\mathbf d}.
\]
The formal miniversal deformation spaces at the two distinguished
points are, respectively, the completion
$\widehat{\mathfrak Z_{\mathbf d}}_{\,0}$ and the formal Kuranishi space of
$E_{\mathbf d}$.

By \cite[Corollary~3.12]{ChenPertusiZhao2024},
$\operatorname{RHom}_{\mathcal C}(E_{\mathbf d},E_{\mathbf d})$ is
formal. Hence the formal Kuranishi space of $E_{\mathbf d}$ is
isomorphic to $\widehat{\mathfrak Z_{\mathbf d}}_{\,0}$. The formality
quasi-isomorphism may be chosen to induce the identity on cohomology,
and therefore the resulting formal isomorphism induces the identity on
the tangent space
$\operatorname{Ext}^1_{\mathcal C}(E_{\mathbf d},E_{\mathbf d})$.
Since $G_{\mathbf d}$ is reductive,
\cite[Proposition~4.7]{ArbarelloSacca2018} allows this formal isomorphism to be
chosen $G_{\mathbf d}$-equivariantly.

By \cite[Theorem~4.19]{AHR20}, there are an affine
$G_{\mathbf d}$-scheme $T$, a fixed point $t\in T$, and pointed étale
morphisms
\[
([T//G_{\mathbf d}],[t])
\longrightarrow
([\mathfrak Z_{\mathbf d}/G_{\mathbf d}],[0]),
\qquad
([T//G_{\mathbf d}],[t])
\longrightarrow
\Bigl(
\mathcal M_\sigma
\bigl(\mathcal C,\operatorname{wt}(\mathbf d)v_0\bigr),
[E_{\mathbf d}]
\Bigr),
\]
inducing isomorphisms on stabilizers at $[t]$. Since the target stacks
have affine diagonal, after shrinking around $[t]$ we may assume that
both morphisms are affine by
\cite[Proposition~3.2]{AHR20}. Applying
\cite[Proposition~4.13]{AHR20} to each morphism and shrinking again,
we may further assume that both maps are cartesian over étale maps of
the corresponding good moduli spaces.

Let $x\in Z_{\mathbf d}$ be simple. Since
$\mu_{\mathbf d}$ is homogeneous of degree $2$, one has
$\lambda x\in Z_{\mathbf d}$ for every $\lambda\in\mathbb C$.
For $\lambda\neq0$, the representations $x$ and $\lambda x$
have the same subrepresentations, so $\lambda x$ remains simple.
Moreover, we have $a_{\lambda x}=\lambda a_x$
and $d\mu_{\mathbf d,\lambda x} = \lambda\,d\mu_{\mathbf d,x}.$ Thus nonzero rescaling does not change the kernels, images, or
cokernels appearing below.

Now let $
f\colon[T//G_{\mathbf d}]
\longrightarrow
[\mathfrak Z_{\mathbf d}//G_{\mathbf d}]$
be the first étale morphism. Its image is an open substack containing
$[0]$. The morphism
\[
\mathbb A^1\longrightarrow[\mathfrak Z_{\mathbf d}//G_{\mathbf d}],
\qquad
\lambda\longmapsto[\lambda x],
\]
sends $0$ to $[0]$. Hence, for some $\lambda\neq0$, the point
$[\lambda x]$ lies in the image of $f$. Choose a lift
$y\in[T//G_{\mathbf d}]$, and denote its image in
$\mathcal M_\sigma
(\mathcal C,\operatorname{wt}(\mathbf d)v_0)$
by $S_x$.

Since $x$ is simple, its $G_{\mathbf d}$-orbit is closed and $\operatorname{Stab}_{G_{\mathbf d}}(x)
=
\Delta\mathbb G_m$. The two étale charts are cartesian over étale maps
of their good moduli spaces. They therefore preserve closed points
and stabilizers. It follows that $S_x$ is polystable and
\[
\operatorname{Aut}_{\mathcal C}(S_x)\simeq\mathbb G_m.
\]
But $\operatorname{Aut}_{\mathcal C}(S_x)$ is isomorphic to $\mathbb G_m$ only when there is one
stable summand with multiplicity one, so $S_x$ is stable. By
construction, its numerical class is
$\operatorname{wt}(\mathbf d)v_0$.
\end{proof}

We next prove the rank estimate needed at stable objects fixed by $\tau$. Let
$\mathcal{D}=\mathcal{C}^{\mathbb Z/2}$ be the CY2 cover of $\mathcal{C}$, and let $\iota\colon \mathcal{D}\to \mathcal{D}$ denote
the residual involution. We write $\operatorname{Forg}\colon D\to \mathcal{C}$ for the forgetful
functor and $\operatorname{Inf}\colon \mathcal{C}\to \mathcal{D}$ for the inflation functor. The induced
stability condition on $\mathcal{D}$ defined as in \cite[Theorem 4.8]{ppzEnriques2021} will be denoted by $\sigma_\mathcal{D}$.

\begin{lemma} \label{lem:rk-equal-ext}
Let $S\in\mathcal C$ be a $\sigma$-stable object such that
$S\simeq\tau S$. Choose a linearization
$\varphi\colon S\xrightarrow{\sim}\tau S$, normalized so that
$\tau(\varphi)\circ\varphi=id_S$. Let $F:=(S,\varphi)$ and
$\iota F:=(S,-\varphi)$ be the two corresponding objects of
$\mathcal D=\mathcal C^{\mathbb Z/2}$. Then:
\begin{enumerate}
    \item $\operatorname{Inf}(S)\simeq F\oplus\iota F$.

    \item The objects $F$ and $\iota F$ are non-isomorphic
    $\sigma_{\mathcal D}$-stable objects of the same phase.

    \item If
    $T_\varphi\colon\Ext^1_{\mathcal C}(S,S)\to
    \Ext^1_{\mathcal C}(S,S)$ is the involution
    $T_\varphi(\alpha):=\varphi^{-1}\circ\tau(\alpha)\circ\varphi$,
    then
    \[
    \Ext^1_{\mathcal C}(S,S)^+
    \simeq \Ext^1_{\mathcal D}(F,F),
    \qquad
    \Ext^1_{\mathcal C}(S,S)^-
    \simeq \Ext^1_{\mathcal D}(F,\iota F).
    \]

    \item Let
    $q_S\colon\Ext^1_{\mathcal C}(S,S)\to
    \Ext^2_{\mathcal C}(S,S)\simeq\mathbb C$
    be the scalar quadratic Kuranishi equation at $S$. Then
    \[
    \rk(q_S)
    =\dim\Ext^1_{\mathcal C}(S,S)^-
    =\ext^1_{\mathcal D}(F,\iota F).
    \]
\end{enumerate}
\end{lemma}

\begin{proof}
The splitting
$\operatorname{Inf}(S)\simeq(S,\varphi)\oplus(S,\varphi\otimes\chi)$
is \cite[Lemma~5.8]{PPZ23}. For $\mathbb Z/2$, the nontrivial
character $\chi$ changes $\varphi$ to $-\varphi$, so this is
precisely
$\operatorname{Inf}(S)\simeq F\oplus\iota F$. Notice that $F$ and $\iota F$ are non-isomorphic. Indeed, an
isomorphism $a\colon F\to\iota F$ would be a scalar automorphism
$a=\lambda id_S$ of the underlying stable object $S$. The
equivariance condition would give
$-\varphi\circ a=\tau(a)\circ\varphi$, and hence
$-\lambda\varphi=\lambda\varphi$, a contradiction. Since
$\operatorname{Forg}(F)=S=\operatorname{Forg}(\iota F)$, both objects have the same central
charge and phase. Any destabilizing subobject of $F$ or $\iota F$
would forget to a destabilizing subobject of the stable object $S$.
Thus $F$ and $\iota F$ are $\sigma_{\mathcal D}$-stable.

The Ext-decomposition follows from the adjunction between $\operatorname{Inf}$
and $\operatorname{Forg}$:
\[
\Ext^1_{\mathcal C}(S,S)
\simeq \Ext^1_{\mathcal D}(F,\operatorname{Inf}(S))
\simeq \Ext^1_{\mathcal D}(F,F)
\oplus \Ext^1_{\mathcal D}(F,\iota F).
\]
Under this decomposition, the first summand is the $+1$-eigenspace
of $T_\varphi$, while the second is the $-1$-eigenspace.

It remains to identify the rank of the scalar quadratic obstruction.
The $T_\varphi$-eigenspace decomposition is compatible with the
Yoneda product. The $+1$-directions are precisely the first-order
deformations of the stable object $F$ in the CY2
category $\mathcal D$. The stable moduli space in a CY2
category is smooth at $F$, so the scalar obstruction vanishes on
these directions. The mixed terms vanish because
$\Ext^2_{\mathcal D}(F,\iota F)=0$ and
$\Ext^2_{\mathcal D}(\iota F,F)=0$ as $F$ and $\iota F$
are non-isomorphic stable objects of the same phase and
$\mathcal D$ is CY2.

Set $V:=\Ext^1_{\mathcal D}(F,\iota F)$, and let
$\operatorname{tr}_F$ and $\operatorname{tr}_{\iota F}$ denote the CY2 trace
maps. Since the equivariant structures of $F$ and $\iota F$ are
$\varphi$ and $-\varphi$, respectively, one has
\[
\operatorname{tr}_{\iota F}\bigl(\iota(\eta)\bigr)=-\operatorname{tr}_F(\eta)
\]
for every $\eta\in\Ext^2_{\mathcal D}(F,F)$. Indeed, both trace maps
are induced from the Serre trace on $\mathcal C$, and replacing the
linearization $\varphi$ by $-\varphi$ changes the induced trace by
a sign.

Define a bilinear form $B\colon V\times V\to\mathbb C$ by
\[
B(b,c):=\operatorname{tr}_F\bigl(\iota(c)\circ b\bigr).
\]
This form is nondegenerate, since it is obtained from the perfect Serre pairing. Moreover, $B$ is symmetric. Indeed, graded cyclicity of the
$2$-Calabi--Yau trace and the preceding trace identity give
\begin{align*}
B(c,b)
&=\operatorname{tr}_F\bigl(\iota(b)\circ c\bigr) \\
&=-\operatorname{tr}_{\iota F}\bigl(c\circ\iota(b)\bigr) \\
&=-\operatorname{tr}_{\iota F}
   \bigl(\iota(\iota(c)\circ b)\bigr) \\
&=\operatorname{tr}_F\bigl(\iota(c)\circ b\bigr)
 =B(b,c).
\end{align*}

Under the identification of $V$ with the $-1$-eigenspace of
$\Ext^1_{\mathcal C}(S,S)$, the restriction of $q_S$ to $V$ is,
up to a nonzero scalar, the quadratic form $b\mapsto B(b,b)$.
Therefore its polar form is a nonzero scalar multiple of
$2B$, and is consequently nondegenerate. Since $q_S$ vanishes on
the $+1$-eigenspace and has no mixed terms, and thus we have the final assertion.
\end{proof}

 We now compare the linearized equations at that point with the deformation theory of the corresponding object.

\begin{lemma}
\label{lem:tangent-obstruction-simple-point}
Let $x\in \mathfrak Z_{\mathbf d}$ be simple, and let $S_x$ be the stable object
obtained from
Lemma~\ref{lem:etale-chart-simple-points}. Let $a_x\colon
\mathfrak g_{\mathbf d}
\longrightarrow
\operatorname{Rep}_{\mathbf d}$ be the infinitesimal action of
$\mathfrak g_{\mathbf d}:=\operatorname{Lie}(G_{\mathbf d})$ at $x$.
Then
\[
\frac{\ker(d\mu_{\mathbf d,x})}{\operatorname{im}(a_x)}
\simeq
\operatorname{Ext}^1_{\mathcal C}(S_x,S_x),
\qquad
\dim\operatorname{coker}(d\mu_{\mathbf d,x})
=
\operatorname{ext}^2_{\mathcal C}(S_x,S_x).
\]
In particular, $d\mu_{\mathbf d,x}$ is surjective when
$S_x\not\simeq\tau S_x$, whereas its cokernel is one-dimensional when
$S_x\simeq\tau S_x$.
\end{lemma}

\begin{proof}
Replace $x$ by the sufficiently small nonzero scalar multiple used in
Lemma~\ref{lem:etale-chart-simple-points}, and continue to denote the
resulting point by $x$. Notice this does not affect the assertion as the kernels, images, and cokernels occurring in the statement do
not change under nonzero rescaling.

Since $\mathfrak Z_{\mathbf d}=\mu_{\mathbf d}^{-1}(0)$, its Zariski tangent
space at $x$ is
$T_x \mathfrak Z_{\mathbf d}=\ker(d\mu_{\mathbf d,x})$. Moreover,
$G_{\mathbf d}$-equivariance of $\mu_{\mathbf d}$ implies
$d\mu_{\mathbf d,x}\circ a_x=0$. Hence the tangent space of the
quotient stack at $[x]$ is
\[
T_{[x]}[\mathfrak Z_{\mathbf d}/G_{\mathbf d}]
\simeq
\frac{\ker(d\mu_{\mathbf d,x})}{\operatorname{im}(a_x)}.
\]
The common étale chart of
Lemma~\ref{lem:etale-chart-simple-points} identifies this tangent
space with the tangent space of the moduli stack at $[S_x]$.
Therefore
\[
\frac{\ker(d\mu_{\mathbf d,x})}{\operatorname{im}(a_x)}
\simeq
T_{[S_x]}
\mathcal M_\sigma
\bigl(\mathcal C,\operatorname{wt}(\mathbf d)v_0\bigr)
\simeq
\operatorname{Ext}^1_{\mathcal C}(S_x,S_x).
\]

To calculate the obstruction dimension, consider the complex
\[
C_x^\bullet
:=
\left[
\mathfrak g_{\mathbf d}
\xrightarrow{\,a_x\,}
\operatorname{Rep}_{\mathbf d}
\xrightarrow{\,d\mu_{\mathbf d,x}\,}
\operatorname{Obs}_{\mathbf d}
\right]
\]
in degrees $0,1,2$. Since $x$ is simple, its stabilizer is the
diagonal scalar subgroup, and hence
$\dim H^0(C_x^\bullet)=1$. The preceding tangent-space calculation
gives $
H^1(C_x^\bullet)
\simeq
\operatorname{Ext}^1_{\mathcal C}(S_x,S_x).$ Stability and pairwise nonisomorphism of the $E_i$ give $\operatorname{Ext}^0_{\mathcal C}(E_{\mathbf d},E_{\mathbf d})
\simeq\mathfrak g_{\mathbf d}$, while
Lemma~\ref{lem:serre-twisted-model} identifies the
degree-one and degree-two Ext groups with
$\operatorname{Rep}_{\mathbf d}$ and
$\operatorname{Obs}_{\mathbf d}$, respectively. The phase and
Serre-duality vanishings show that there are no self-Ext groups outside
degrees $0,1,2$. Thus
\[
\chi(C_x^\bullet)
=
\dim\mathfrak g_{\mathbf d}
-\dim\operatorname{Rep}_{\mathbf d}
+\dim\operatorname{Obs}_{\mathbf d}
=
\chi(E_{\mathbf d},E_{\mathbf d}).
\]

The objects $E_{\mathbf d}$ and $S_x$ have the same numerical class,
so their Euler characteristics agree. Since $S_x$ is stable and has
no self-Ext groups outside degrees $0,1,2$, we have
\[
\chi(S_x,S_x)
=
1-\operatorname{ext}^1_{\mathcal C}(S_x,S_x)
+\operatorname{ext}^2_{\mathcal C}(S_x,S_x).
\]
Comparing this with $\chi(C_x^\bullet)
=
\dim H^0(C_x^\bullet)
-\dim H^1(C_x^\bullet)
+\dim H^2(C_x^\bullet)$
gives $\dim H^2(C_x^\bullet) = \operatorname{ext}^2_{\mathcal C}(S_x,S_x).$ Since the complex terminates in $\operatorname{Obs}_{\mathbf d}$,
one has
$H^2(C_x^\bullet)=\operatorname{coker}(d\mu_{\mathbf d,x})$, proving
the second assertion.
\end{proof}

\begin{lemma}
\label{lem:local-fixed-simple}
Let $x\in \mathfrak Z_{\mathbf d}$ be simple, and let $S_x$ be the stable
object obtained from Lemma~\ref{lem:etale-chart-simple-points}. Assume that
$S_x\simeq\tau S_x$. Set
\[
h:=\operatorname{ext}^1_{\mathcal C}(S_x,S_x),
\qquad
\rho:=\operatorname{rk}(q_{S_x}),
\]
where $q_{S_x}$ is the scalar quadratic Yoneda square. Then there is
a noncanonical isomorphism
\[
\widehat{\mathcal O}_{Z_{\mathbf d},x}
\simeq
\frac{
\mathbb C[[u_1,\ldots,u_h,
z_1,\ldots,z_{\dim G_{\mathbf d}-1}]]
}{
(u_1^2+\cdots+u_\rho^2)
}.
\]
\end{lemma}

\begin{proof}
After the nonzero rescaling used in
Lemma~\ref{lem:etale-chart-simple-points}, which does not change the completed local
ring, we may assume that $x$ lies in the common étale neighborhood.

Choose a complement
$\operatorname{Rep}_{\mathbf d}=\operatorname{im}(a_x)\oplus N$. Since $\Delta\mathbb G_m$ acts trivially on
$\operatorname{Rep}_{\mathbf d}$, the morphism
\[
G_{\mathbf d} \times^{\Delta\mathbb G_m}(x+N)\longrightarrow \operatorname{Rep}_{\mathbf d},
\qquad
[g,y]\longmapsto g\cdot y,
\]
is étale at $[1,x]$. Setting $Y=(x+N)\cap \mathfrak Z_{\mathbf d}$, we
therefore obtain
\[
\widehat{\mathcal O}_{\mathfrak Z_{\mathbf d},x}
\simeq
\widehat{\mathcal O}_{Y,x}
[[z_1,\ldots,z_{\dim G_{\mathbf d}-1}]].
\]

The common étale chart of Lemma~\ref{lem:etale-chart-simple-points} identifies the
completed germ of $[Y/\Delta\mathbb G_m]$ with the completed moduli-stack germ at
$[S_x]$. Since $\Delta\mathbb G_m$ acts trivially on $Y$, and scalar
automorphisms act trivially by conjugation on
$R\!\Hom_{\mathcal C}(S_x,S_x)$,
$\widehat{\mathcal O}_{Y,x}$ is the miniversal deformation ring of
$S_x$.

Now $S_x\simeq\tau S_x$, so
$\operatorname{Ext}^2_{\mathcal C}(S_x,S_x)\simeq\mathbb C$.
By \cite[Corollary~3.12]{ChenPertusiZhao2024}, the derived endomorphism algebra of
$S_x$ is formal. Hence
\[
\widehat{\mathcal O}_{Y,x}
\simeq
\mathbb C[[u_1,\ldots,u_h]]/(q_{S_x}).
\]
After a linear change of coordinates over $\mathbb C$, the quadratic
form $q_{S_x}$ becomes
$u_1^2+\cdots+u_\rho^2$, which proves the claim.
\end{proof}

\subsection{Rank of Kuranishi equation} The following lemma gives out a lower bout of the rank of Kuranishi map at $S$, which allow us to estimate the codimension of the singular locus of $\mathfrak Z_\mathbf d$ combining Lemma \ref{lem:local-fixed-simple}.

\begin{lemma}\label{lem:rank-estimate-fixed}
Let $S\in \mathcal{C}$ be a $\sigma$-stable object such that $S\simeq\tau S$. Assume
$[S]=mv_0$, $m\ge 1$, where $v_0\in \mathrm{K}_{\mathrm{num}}(\mathcal{C})$ is primitive and
$v_0^2\ge 5$. Then the scalar quadratic Kuranishi equation at $S$ has rank at
least $3$.
\end{lemma}

\begin{proof}
Let $F,\iota F\in\mathcal D$ be the two equivariant lifts of $S$
from Lemma~4.9, and write
\[
v:=[S]=mv_0,\qquad \lambda:=[F],\qquad
r:=\operatorname{rk}(q_S).
\]
By Lemma~\ref{lem:rk-equal-ext}, $r=\operatorname{ext}^1_{\mathcal D}(F,\iota F)$.
Since $F$ and $\iota F$ are non-isomorphic stable objects of the
same phase, stability and the CY$2$ property give $
\Hom_{\mathcal D}(F,\iota F) = \operatorname{Ext}^2_{\mathcal D}(F,\iota F) =0.$
Consequently,
$r=-\chi_{\mathcal D}(F,\iota F)=(\lambda,\iota\lambda)$.

Set $\delta:=\lambda-\iota\lambda$. The adjunction between
$\operatorname{Inf}$ and $\operatorname{Forg}$, together with
$\operatorname{Inf}(S)\simeq F\oplus\iota F$, gives
$v^2=\lambda^2+r$. Since $\iota$ preserves the Mukai pairing,
$\delta^2=2\lambda^2-2r$. Eliminating $\lambda^2$, we obtain
\[
4r=2v^2-\delta^2.
\]

We claim that $\delta^2\leq 0$. By
\cite[Theorem~4.15(3)]{BP23}, the CY$2$ cover
$\mathcal D$ is equivalent to $\operatorname{Ku}(X^{\mathrm{op}})$,
where $X^{\mathrm{op}}$ is either a K3 surface or a GM fourfold.
In either case, $\mathcal D$ carries a Mukai Hodge structure
$\widetilde H(\mathcal D,\mathbb Z)$ whose underlying lattice has
signature $(4,20)$; see \cite[Examples~5.1--5.2]{BP23} and
\cite[\S3.1]{Per19}. Its real $(1,1)$-part therefore has signature
$(2,20)$.

Consider
$P:=\operatorname{Inf}_*
(\mathrm{K}_{\mathrm{num}}(\mathcal C)_{\mathbb R})$.
For $a,b\in \mathrm{K}_{\mathrm{num}}(\mathcal C)$, adjunction and
$\operatorname{Forg}\circ\operatorname{Inf}
\simeq\operatorname{id}\oplus\tau$ give
\[
(\operatorname{Inf}_*a,\operatorname{Inf}_*b)_{\mathcal D}
=2(a,b)_{\mathcal C},
\]
as in Lemma~\ref{lem:inflation-pairing}.
Thus $P$ is a positive-definite two-plane. The residual involution
$\iota$ fixes $P$, whereas
$\iota(\delta)=-\delta$. Since $\iota$ preserves the Mukai pairing,
we obtain $\delta\perp P$. Moreover, $\delta$ is an algebraic Mukai class, hence
$\delta\in\widetilde H^{1,1}(\mathcal D,\mathbb R)$. Since this
space has positive index $2$ and already contains the positive
two-plane $P$, its orthogonal complement $P^\perp$ is negative
definite. Therefore $\delta^2\leq 0$. Now it follows from $4r=2v^2-\delta^2$ that $r\geq v^2/2$.
Since $v_0^2 \ge 5$, we have  $r\geq 5/2$. As $r$ is an integer,
$\operatorname{rk}(q_S)=r\geq 3$.
\end{proof}

\begin{remark}
Notice $\mathrm{K}_{\mathrm{num}}(\mathcal C)\simeq\mathbb Z^2$ with quadratic form
$x^2+y^2$, and no primitive class has square $3$ or $4$.
Thus the assumption $v_0^2\geq 5$ equivalent to $v_0^2\geq 3$.
\end{remark}

\subsection{Normality of the Serre-twisted zero fibre}

The following proposition shows the normality of the local model.
\begin{proposition}\label{prop:zero-fibre-normal}
If $v_0^2 \ge 5$ where $v_0$ is primitive, then for every dimension vector $\mathbf d=(d_i)_{i\in I}$, the scheme
$\mathfrak Z_{\mathbf d}=\mu_{\mathbf d}^{-1}(0)$ is normal.
\end{proposition}

\begin{proof}
We prove the stronger statement that $\mathfrak Z_{\mathbf d}$ is a reduced local
complete intersection of dimension $\operatorname{expdim}(\mathbf d)$, regular in
codimension one. Normality then follows from Serre's criterion. We argue by induction on $|\mathbf d|$. The case
$\mathbf d=0$ is immediate, and when $|\mathbf d|=1$ every representation is simple.

Since $\mathfrak Z_{\mathbf d}\subset\operatorname{Rep}_{\mathbf d}$ is cut out by
$\dim\operatorname{Obs}_{\mathbf d}$ equations, every irreducible component of
$\mathfrak Z_{\mathbf d}$ has dimension at least $\operatorname{expdim}(\mathbf d)$.
We prove the reverse inequality by showing that every component meets the simple
locus. Let $0<\boldsymbol\beta<\mathbf d$ be a nonzero proper sub-dimension vector, and set
$\boldsymbol\gamma:=\mathbf d-\boldsymbol\beta$. Consider the locus of points of
$\mathfrak Z_{\mathbf d}$ preserving subspaces $U_i'\subset U_i$ with
$\dim U_i'=\beta_i$. For fixed $U_i'$, the block-upper-triangular part of
$\operatorname{Rep}_{\mathbf d}$ has dimension
\begin{equation*}
\dim\operatorname{Rep}_{\boldsymbol\beta}
+\dim\operatorname{Rep}_{\boldsymbol\gamma}
+\sum_{i,j}a_{ij}\gamma_i\beta_j.
\end{equation*}
Ignoring all off-diagonal equations only enlarges the locus. By induction on
$|\mathbf d|:=\sum_i d_i$, the zero loci for the subobject and quotient have
dimensions at most $\operatorname{expdim}(\boldsymbol\beta)$ and
$\operatorname{expdim}(\boldsymbol\gamma)$, respectively. Thus, for fixed $U_i'$, the
locus has dimension at most
\begin{equation*}
\operatorname{expdim}(\boldsymbol\beta)+\operatorname{expdim}(\boldsymbol\gamma)
+\sum_{i,j}a_{ij}\gamma_i\beta_j.
\end{equation*}
Varying the subspaces $U_i'\subset U_i$ adds $\sum_i\beta_i\gamma_i$ parameters.
Hence the locus preserving a sub-dimension vector $\boldsymbol\beta$ has dimension at
most
\begin{equation*}
\operatorname{expdim}(\boldsymbol\beta)+\operatorname{expdim}(\boldsymbol\gamma)
+\sum_{i,j}a_{ij}\gamma_i\beta_j+\sum_i\beta_i\gamma_i.
\end{equation*}
We compare this with $\operatorname{expdim}(\mathbf d)$. Using
$a_{ij}=v_0^2m_i m_j+\delta_{ij}+\delta_{j,\nu(i)}$, one computes
\begin{align*}
&\operatorname{expdim}(\mathbf d)-\left(
\operatorname{expdim}(\boldsymbol\beta)+\operatorname{expdim}(\boldsymbol\gamma)
+\sum_{i,j}a_{ij}\gamma_i\beta_j+\sum_i\beta_i\gamma_i\right)\\
&\hspace{2cm}=v_0^2\operatorname{wt}(\boldsymbol\beta)\operatorname{wt}(\boldsymbol\gamma)
-\sum_i\beta_i\gamma_{\nu(i)}.
\end{align*}
Since $m_i\ge 1$ and $m_{\nu(i)}=m_i$, we have
\begin{equation*}
\sum_i\beta_i\gamma_{\nu(i)}\le
\left(\sum_i\beta_i\right)\left(\sum_i\gamma_i\right)
\le \operatorname{wt}(\boldsymbol\beta)\operatorname{wt}(\boldsymbol\gamma).
\end{equation*}
Therefore the codimension gap is at least
\begin{equation*}
(v_0^2-1)\operatorname{wt}(\boldsymbol\beta)\operatorname{wt}(\boldsymbol\gamma)
\ge 2.
\end{equation*}
It follows that the non-simple locus has dimension at most
$\operatorname{expdim}(\mathbf d)-2$. Since every irreducible component of
$\mathfrak Z_{\mathbf d}$ has dimension at least $\operatorname{expdim}(\mathbf d)$,
no component is contained in the non-simple locus. Hence every irreducible component
meets the simple locus.

Let $x\in\mathfrak Z_{\mathbf d}$ be a simple point. By
Lemma~\ref{lem:etale-chart-simple-points}, after rescaling $x$ if necessary, $x$
corresponds to a $\sigma$-stable object $S_x$ of class
$\operatorname{wt}(\mathbf d)v_0$, and
\begin{equation*} \dim
\operatorname{coker}(d\mu_{\mathbf d,x})\simeq\operatorname{ext}^2_C(S_x,S_x)
\simeq\Hom_C(S_x,\tau S_x)^\vee.
\end{equation*}
If $S_x\not\simeq\tau S_x$, then 
$\Hom_C(S_x,\tau S_x)=0$. Thus $d\mu_{\mathbf d,x}$ is surjective,
and $\mathfrak Z_{\mathbf d}$ is smooth of dimension $\operatorname{expdim}(\mathbf d)$
at $x$.

Now suppose $S_x\simeq\tau S_x$. Then $\operatorname{Ext}^2_C(S_x,S_x)\simeq
\mathbb C$. By Lemma~\ref{lem:local-fixed-simple}, the completed local ring of
$\mathfrak Z_{\mathbf d}$ at $x$ is, up to the smooth orbit factor
$G_{\mathbf d}/\mathbb G_m$, the completed local ring of the hypersurface
$q_{S_x}=0$ inside $\operatorname{Ext}^1_C(S_x,S_x)$. By
Lemma~\ref{lem:rank-estimate-fixed}, the quadratic form $q_{S_x}$ has rank at least
$3$. After a formal linear change of coordinates, the hypersurface factor has the
form
\begin{equation*}
\mathbb C[[u_1,\ldots,u_h]]/(u_1^2+\cdots+u_\rho^2),
\qquad \rho\ge 3.
\end{equation*}
This hypersurface is reduced. Its singular locus is cut out by
$u_1=\cdots=u_\rho=0$, so its codimension inside the hypersurface is
$\rho-1\ge 2$. Hence $\mathfrak Z_{\mathbf d}$ is reduced and regular in
codimension one at such a simple point, and its local dimension is
$\operatorname{expdim}(\mathbf d)$.

We have shown that every irreducible component of $\mathfrak Z_{\mathbf d}$ contains
a simple point at which the local dimension is $\operatorname{expdim}(\mathbf d)$.
Therefore every
component has dimension exactly $\operatorname{expdim}(\mathbf d)$. Thus
$\mathfrak Z_{\mathbf d}\subset\operatorname{Rep}_{\mathbf d}$ has codimension
$\dim\operatorname{Obs}_{\mathbf d}$. As it is cut out by
$\dim\operatorname{Obs}_{\mathbf d}$ equations in the smooth affine space
$\operatorname{Rep}_{\mathbf d}$, those equations form a regular sequence. Hence
$\mathfrak Z_{\mathbf d}$ is a local complete intersection.

The simple locus is dense in every irreducible component and is reduced at every
simple point by the preceding analysis. Thus $\mathfrak Z_{\mathbf d}$ is generically
reduced. A local complete intersection is Cohen--Macaulay, hence has no embedded
associated points; therefore generic reducedness implies reducedness. Finally, $\mathfrak Z_{\mathbf d}$ is Cohen--Macaulay, so it satisfies Serre's
condition $S_2$. By previous calculation, its non-simple locus has codimension at least $2$, and the singular loci is
rank-$\rho$ quadratic hypersurfaces above which has codimension $\rho-1$ inside the hypersurface. Hence
$\mathfrak Z_{\mathbf d}$ satisfies $R_1$. This shows
$\mathfrak Z_{\mathbf d}$ is normal.
\end{proof}

\subsection{Normality of the Bridgeland moduli space} 

\begin{theorem}\label{thm:normality-main}
Let $X$ be a smooth Gushel--Mukai threefold and set $\mathcal{C}=\Ku(X)$. Let $\sigma$ be
a Serre-invariant stability condition on $C$. Let $v_0\in \mathrm{K}_{\mathrm{num}}(\mathcal{C})$ be
primitive with $v_0^2\ge 5$. Then, for every $m\ge 1$, the moduli space
$M_\sigma(\mathcal{C},mv_0)$ is normal.
\end{theorem}

\begin{proof}
Let $[E]\in M_\sigma(\mathcal{C},mv_0)$ be a closed point, represented by a
$\sigma$-polystable object $E\simeq\bigoplus_{i\in I}E_i\otimes V_i$, where the
$E_i$ are pairwise non-isomorphic stable objects of the same phase. By
Remark~\ref{rmk:injective-Z}, every stable factor satisfies $[E_i]=m_i v_0$
for some $m_i>0$.

By Lemma~\ref{lem:serre-twisted-model}, there is an analytic local isomorphism
\begin{equation*}
\bigl(M_\sigma(\mathcal{C},mv_0),[E]\bigr)\simeq
\bigl(\mathfrak Z_{\mathbf n}//G_{\mathbf n},0\bigr),
\end{equation*}
where $\mathbf n=(\dim V_i)_{i\in I}$. By
Proposition~\ref{prop:zero-fibre-normal}, the affine scheme $\mathfrak Z_{\mathbf n}$
is normal. Since $G_{\mathbf n}$ is reductive over $\mathbb{C}$, the invariant ring
$\mathbb{C}[\mathfrak Z_{\mathbf n}]^{G_{\mathbf n}}$ is normal. Therefore the good
quotient $\mathfrak Z_{\mathbf n}//G_{\mathbf n}$ is normal. Hence the analytic germ
of $M_\sigma(\mathcal{C},mv_0)$ at $[E]$ is normal. Normality is local in the analytic
topology, so $M_\sigma(\mathcal{C},mv_0)$ is normal.
\end{proof}

\section{Irreducibility of singular normal moduli spaces}

Let us first introduce a lemma which introduces the criterion of existence of singularities. 
\begin{lemma}\label{lem:tau–fixed–implies–singular}
Let $X$ be a Gushel–Mukai threefold, $\Ku(X)$ its Kuznetsov component with involution $\tau$, and
\[
M \;=\; M_{\sigma}\bigl(\mathcal{C},v_0\bigr)
\]
the Bridgeland moduli space of $\sigma$–stable objects of a primitive class $v_0$.  Suppose $E\in \mathcal{C}$ is $\sigma$–stable, then
$\tau(E)\cong E$ if and only if the point $[E]\in M$ is a singular point of $M$.
\end{lemma}

\begin{proof}
One direction is shown in \cite[Theorem 1.8(1)]{PPZ23}, so we are left to show another.  First, as $M$ is reduced, it is generic smooth and therefore there exists at least one smooth point. Suppose $E$ is fixed by $\tau$. Notice that by Serre duality one has $\ext^0(E,E) = \ext^2(E,E) = 1 $. However, assuming $E$ is a smooth point, the expected dimension of $M$ is (\cite[proposition 5.1]{FGLZ25})
\[1 - \chi(v,v) = 1 - (2 - \ext^1(E,E)) =  \ext^1(E,E) - 1,\]
which is smaller then the actual dimension of $M$. Hence $[E]$ has to be singular.
\end{proof}

 The goal of the following theorem is to show the connectedness of singular normal moduli spaces of the category $\mathcal{C}$. To do so, we want to modify the classic proof in \cite[Section~4]{kaledin06Sin}.

\begin{theorem}
\label{thm:singular-connectedness}
Let $X$ be a smooth Gushel--Mukai threefold, let
$\mathcal C=\operatorname{Ku}(X)$, and let $\sigma$ be a
Serre-invariant stability condition on $\mathcal C$. Let
$v\in K_{\mathrm{num}}(\mathcal C)$ be primitive with $v^2$ odd,
and assume that
\[
M:=M_\sigma(\mathcal C,v)
\]
is normal. If $M$ contains a singular point, then $M$ is connected
and hence irreducible.
\end{theorem}

We prepare the relative complex used in the proof. Let $Y\subset M$ be a connected component containing a
singular point $p=[F]$. By Lemma \ref{lem:tau–fixed–implies–singular}, one has $F\simeq\tau F$.

Since the Euler lattice $K_{\mathrm{num}}(\mathcal C)$ is unimodular
and $v$ is primitive, there exists
$w\in K_{\mathrm{num}}(\mathcal C)$ such that $\chi(w,v)=1$.
The Brauer-class argument in
\cite[proof of Corollary~3.9]{LLPZ24} shows that $M$ is a fine
moduli space. Indeed, the universal twisted object on the moduli stack
produces a twisted perfect complex of rank $\chi(w,v)=1$, forcing
the Brauer class of the moduli gerbe to vanish. Thus there is a
universal perfect complex
$\mathcal E\in D_{\mathrm{perf}}(Y\times X)$.

Let
$\pi_Y\colon Y\times X\to Y$ and
$\pi_X\colon Y\times X\to X$ be the projections, and set
\[
\mathcal K_F:=
\mathbf R\pi_{Y*}\mathbf R\mathcal Hom
\bigl(\pi_X^*F,\mathcal E\bigr).
\]
For every $y\in Y$, the fiber of $\mathcal K_F$ computes
$\operatorname{Ext}^\bullet_{\mathcal C}(F,\mathcal E_y)$, and notice these groups vanish outside degrees $0,1,2$. On a neighborhood of $p$, choose a three-term locally
free presentation, compatible with base change,
\begin{equation}
\label{eq:relative-ext-complex}
\mathcal K_F\simeq
\bigl[
A^0\xrightarrow{\alpha}A^1\xrightarrow{\beta}A^2
\bigr].
\end{equation}
Now we need one extra lemma to describe the degeneracy loci.

\begin{lemma}
\label{lem:reduced-degeneracy}
In the notation above, the degeneracy loci of $\alpha$ and $\beta$
are both the reduced point $\{p\}$.
\end{lemma}

\begin{proof}
Let $y\in Y\setminus\{p\}$. The objects $F$ and
$\mathcal E_y$ are non-isomorphic stable objects of the same class
and phase, so
$\operatorname{Hom}_{\mathcal C}(F,\mathcal E_y)=0$. On the other hand $\operatorname{Ext}^2_{\mathcal C}(F,\mathcal E_y)
\simeq
\operatorname{Hom}_{\mathcal C}(\mathcal E_y,\tau F)^\vee =0.$ Thus $\alpha$ is injective and $\beta$ is surjective away from
$p$. Also, $\alpha$ and $\beta$ each drop rank by exactly one at $p$. This proves the set-theoretic assertion. Moreover, one has $\rk\alpha(p)=\rk A^0-1$ and $\rk\beta(p)=\rk A^2-1.$

It remains to prove that the two degeneracy loci are reduced. Let
$0\ne\gamma\in\operatorname{Ext}^1_{\mathcal C}(F,F)$, and let
$F_\gamma$ be the corresponding first-order deformation over
$R:=\mathbb C[\epsilon]/(\epsilon^2)$. It fits into the usual
extension
\begin{equation}
\label{eq:first-order-extension}
0\rightarrow F
\xrightarrow{\epsilon}F_\gamma
\rightarrow F\rightarrow0.
\end{equation}
The induced long exact sequence in Ext has boundary maps given, up to
the usual sign convention, by
$\gamma\cup-$. The map $\gamma\cup-\colon
\operatorname{Hom}_{\mathcal C}(F,F)
\longrightarrow
\operatorname{Ext}^1_{\mathcal C}(F,F)$
is clearly injective, since it sends $id_F$ to $\gamma$. We claim that
\[
\gamma\cup-\colon
\operatorname{Ext}^1_{\mathcal C}(F,F)
\longrightarrow
\operatorname{Ext}^2_{\mathcal C}(F,F)
\]
is surjective. Choose an isomorphism
$\theta\colon F\xrightarrow{\sim}\tau F$. Under this identification,
Serre duality gives a perfect bilinear form
\[
B_\theta(\gamma,\eta)
:=
\operatorname{tr}_F
\bigl(\theta[2]\circ(\gamma\cup\eta)\bigr)
\]
on $\operatorname{Ext}^1_{\mathcal C}(F,F)$. Since $\gamma\ne0$,
there exists $\eta$ such that $B_\theta(\gamma,\eta)\ne0$.
Therefore $\gamma\cup\eta\ne0$. As
$\operatorname{Ext}^2_{\mathcal C}(F,F)\simeq\mathbb C$, the map
$\gamma\cup-$ is surjective.

It follows from the long exact sequence associated with
\eqref{eq:first-order-extension} that
\[
\operatorname{Ext}^0_R
(F\boxtimes R,F_\gamma)
\simeq\operatorname{Hom}_{\mathcal C}(F,F)
\simeq\mathbb C
\]
and
\[
\operatorname{Ext}^2_R
(F\boxtimes R,F_\gamma)
\simeq\operatorname{Ext}^2_{\mathcal C}(F,F)
\simeq\mathbb C.
\]
Thus no nonzero tangent vector at $p$ lies in the tangent space of
either degeneracy locus, otherwise the relative Ext
group would have length greater than one. Hence both degeneracy loci
have zero Zariski tangent space at $p$. Since they are supported at
the single point $p$, Nakayama's lemma implies that they are the
reduced point $\{p\}$.
\end{proof}

\begin{proof}[Proof of Theorem~\ref{thm:singular-connectedness}]
Assume, for contradiction, that $M$ is disconnected, and retain the
component $Y$ and singular point $p=[F]$ fixed above.

We first record the dimension. Since $M$ is normal, every
irreducible component is generically smooth. At a smooth point
$[E]$, Lemma \ref{lem:tau–fixed–implies–singular} gives $E\not\simeq\tau E$, and hence
$\operatorname{Ext}^2_{\mathcal C}(E,E)=0$. Therefore every
component has dimension
\[
d:= \operatorname{ext}^1_{\mathcal C}(E,E) = v^2+1.
\]
In particular, $d$ is even because $v^2$ is odd. 
Since $F\simeq\tau F$, the automorphism of $M$ induced by $\tau$
fixes $p$, and therefore preserves $Y$. Choose a stable object
$G$ represented by a point of $M\setminus Y$. Then $\tau G$
also lies outside $Y$. Hence, for every $y\in Y$, we have
$\operatorname{Hom}_{\mathcal C}(G,\mathcal E_y)=0$ and $
\operatorname{Ext}^2_{\mathcal C}(G,\mathcal E_y)
\simeq
\operatorname{Hom}_{\mathcal C}(\mathcal E_y,\tau G)^\vee=0.$ Consequently, if
\[
\mathcal K_G:=
\mathbf R\pi_{Y*}\mathbf R\mathcal Hom
\bigl(\pi_X^*G,\mathcal E\bigr),
\]
then $\mathcal K_G\simeq\mathcal W[-1]$ for a vector bundle
$\mathcal W$ on $Y$ of rank $\operatorname{rk}\mathcal W
=-\chi(G,\mathcal E_y)=v^2=d-1.$ For $y\ne p$, the similar argument gives
$\mathcal K_F|_{Y\setminus\{p\}}\simeq\mathcal V[-1]$ for a vector bundle $\mathcal V$ of rank $d-1$.

Let $\varphi\colon Z:=\operatorname{Bl}_p(Y)\to Y$ be the blow-up
at $p$, let $D\subset Z$ be the exceptional divisor, and let
$i\colon D\hookrightarrow Z$ be the inclusion. Put
$\widetilde U:=\varphi^{-1}(U)$ and
$P:=\mathbf L\varphi^*\mathcal K_F$.

By Lemma~\ref{lem:reduced-degeneracy}, the maximal-minor ideals of
$\alpha$ and $\beta$ are both the maximal ideal
$\mathfrak m_p$. Their pullbacks to $\widetilde U$ are therefore
equal to $\mathcal O_{\widetilde U}(-D)$. Since the two ranks drop
by exactly one, the elementary modification of
\cite[Section~4]{kaledin06Sin} gives locally free sheaves
$A^{\prime0}$ and $A^{\prime2}$ and a complex
\begin{equation}
\label{eq:modified-complex}
C':=
\bigl[
A^{\prime0}\xrightarrow{\alpha'}
\varphi^*A^1\xrightarrow{\beta'}A^{\prime2}
\bigr],
\end{equation}
where $A^{\prime0}\subset\varphi^*A^1$ is a subbundle,
$\varphi^*A^1\to A^{\prime2}$ is surjective, and
\[
\varphi^*A^0\subset A^{\prime0},
\qquad
A^{\prime2}\subset\varphi^*A^2.
\]

There are line bundles $\mathcal L,\mathcal N$ on $D$ determined
by
\[
i_*\mathcal L=\varphi^*A^2/A^{\prime2},
\qquad
i_*\mathcal N=A^{\prime0}/\varphi^*A^0.
\]
Derived base change along $D\to p$ gives $\mathcal L\simeq
\operatorname{Ext}^2_{\mathcal C}(F,F)\otimes\mathcal O_D
\simeq\mathcal O_D.$ Moreover,
\[
\mathcal N\otimes\mathcal O_D(-D)
\simeq
\operatorname{Tor}_1^Z(i_*\mathcal N,\mathcal O_D)
\simeq
\operatorname{Hom}_{\mathcal C}(F,F)\otimes\mathcal O_D
\simeq\mathcal O_D,
\]
and hence $\mathcal N\simeq\mathcal O_D(D)$.

We now globalize the modification. On $\widetilde U$, set
\[
C=
\bigl[\varphi^*A^0\to\varphi^*A^1\to\varphi^*A^2\bigr],
\qquad
C_0=
\bigl[A^{\prime0}\to\varphi^*A^1\to\varphi^*A^2\bigr].
\]
There are short exact sequences of complexes
\begin{equation}
\label{eq:local-modification-sequences}
0\longrightarrow C\longrightarrow C_0
\longrightarrow i_*\mathcal N\longrightarrow0,
\qquad
0\longrightarrow C'\longrightarrow C_0
\longrightarrow i_*\mathcal L[-2]\longrightarrow0.
\end{equation}
Over $\widetilde U\setminus D$, all three complexes agree with
$P$. Gluing $C_0$ and $C'$ with $P|_{Z\setminus D}$ gives
perfect complexes $P_0$ and $P'$ on $Z$, fitting into
distinguished triangles
\[
P\longrightarrow P_0\longrightarrow i_*\mathcal N
\longrightarrow P[1],
\qquad
P'\longrightarrow P_0\longrightarrow i_*\mathcal L[-2]
\longrightarrow P'[1].
\]
The first map in $C'$ is a subbundle embedding and the second is
surjective, so $C'$ has a single locally free cohomology sheaf in
degree one. Together with
$P|_{Z\setminus D}\simeq\varphi^*\mathcal V[-1]$, this shows that
$P'\simeq\mathcal W'[-1]$ for a vector bundle $\mathcal W'$ on
$Z$ of rank $d-1$.

Taking classes in $K_0^{\mathrm{perf}}(Z)$ in the two triangles
gives
\begin{equation}
\label{eq:modified-k-class}
[\mathcal W']
=
-[\mathbf L\varphi^*\mathcal K_F]
+[i_*\mathcal O_D]-[i_*\mathcal O_D(D)].
\end{equation}

Because $F$ and $G$ have the same numerical class, the explicit
description of $K_{\mathrm{num}}(\mathcal C)$ implies that they have
the same rational Chern character. Relative
Grothendieck--Riemann--Roch therefore gives
$\operatorname{ch}(\mathcal K_F)=\operatorname{ch}(\mathcal K_G)$.
Since $\mathcal K_G\simeq\mathcal W[-1]$, it follows that
$c(-\mathcal K_F)=c(\mathcal W)$.

Taking Chern classes in $A^*(Z)_{\mathbb Q}$ in
\eqref{eq:modified-k-class}, we obtain
\[
c(\mathcal W')
=
\varphi^*c(\mathcal W)\,
\frac{c(i_*\mathcal O_D)}
     {c(i_*\mathcal O_D(D))}.
\]
The standard exact sequences associated with $D$ give
\begin{equation}
\label{eq:correction-chern-class}
\frac{c(i_*\mathcal O_D)}
     {c(i_*\mathcal O_D(D))}
=
\frac{1}{(1-D)(1+D)}
=
1+D^2+D^4+\cdots.
\end{equation}

There are no mixed terms in positive degree. Indeed, for $j,k>0$,
the projection formula gives
\[
\varphi^*c_j(\mathcal W)\cdot D^k
=
i_*\!\left(
c_j(i^*\varphi^*\mathcal W)
c_1(\mathcal O_D(D))^{k-1}
\right)=0,
\]
because $i^*\varphi^*\mathcal W
\simeq\mathcal W_p\otimes\mathcal O_D$ is trivial. Since $d$ is
even, the codimension-$d$ part of
\eqref{eq:correction-chern-class} gives
\begin{equation}
\label{eq:top-chern-contradiction}
c_d(\mathcal W')
=
\varphi^*c_d(\mathcal W)+D^d.
\end{equation}
Both $\mathcal W$ and $\mathcal W'$ have rank $d-1$, so their
$d$-th Chern classes vanish. Thus
\eqref{eq:top-chern-contradiction} implies $D^d=0$. Therefore
\[
\deg(D^d)
=
\int_D c_1\bigl(\mathcal O_D(D)\bigr)^{d-1}
=
(-1)^{d-1}\operatorname{mult}_p(Y)\ne0,
\]
a contradiction. Hence $M$ is connected. Since $M$ is normal by
\cite{FGLZ25}[Proposition 5.1], its irreducible components are open and closed, and
therefore $M$ is irreducible.
\end{proof}

\begin{remark}
When $v^2$ is even, $d$ is odd and the
correction factor has no codimension-$d$ term. The first
nontrivial identity occurs in codimension $d-1$, where the
Chern classes of the rank-$(d-1)$ bundles need not vanish.
Thus the argument does not decide connectedness for primitive
classes of even square.
\end{remark}

\section{Kuznetsov component of Gushel-Mukai threefold}
In this section, we combine the normality and connectedness results of the previous two sections to prove the main theorem. We first establish a deformation lemma for proper families with normal equidimensional fibres, and then construct a singular irreducible special fibre whose irreducibility deforms to a general Gushel--Mukai threefold.

\subsection{An Deformation Lemma}

\begin{lemma} \label{an_ag_lemma}
    Let $M$ be a scheme, and $C$ an irreducible smooth curve. If $\pi: M \rightarrow C$ is a proper morphism such that 
    \begin{enumerate}
        \item each fibre of $\pi$ is normal and of equidimension $d$;
        \item there exists a point $a \in C$ whose fibre $M_a$ is irreducible,
    \end{enumerate}
    then there exists an open subset of $C$ containing $a$ whose fibres are irreducible.
\end{lemma}

\begin{proof}
    First notice all the fibres of $\pi$ are reduced, so we can replace $M$ with $M_{red}$ without changing fibres. By hypothesis (1), $\pi$ is surjective.
    
    Now notice that for each irreducible component in $M$, its image under $\pi$ is either the entire $C$ or a point. Denote by $M'$ the union of all irreducible components of $M$ that do not dominate $C$. Therefore, the restriction of $\pi$ to $\pi^{-1}(C-S)$ is flat, where $S:=\pi(M')$. As fibres are normal, the number of connected components of fibre should be constant over $C-S$. If $a \notin M'$, then the existence of a connected fibre implies every $M_c$ where $c \in C-S$ is connected. But all these fibres are normal, so they are moreover irreducible.
    
    The only problem here is to exclude the case when $a\in S$. So, let us suppose there exists an irreducible component $Z$ with $\pi(Z) = a$. Take $\eta$ to be the generic point of $C$, and denote $Y:= \overline{M_\eta}$. As $\pi|_{Y}$ is proper and contains $\eta$ in its image, $\pi|_Y$ dominates $C$. Thus, $\pi|_Y$ is flat and therefore its fibres are of equidimension $d$. Note that this shows $Y \cap Z \ne \emptyset$, and $\dim Z = d$ by hypothesis. On the other hand, $\dim (Y\cap Z) = \dim \pi|_Y^{-1}(a) = d$ by previous calculation, so we have $Z \subset Y$ as $Z$ is irreducible. This is absurd since $Z$ is assumed to be an irreducible component of $M$.
\end{proof}

Take the CY2 cover $D^b(S)$ of $\Ku(X)$ where X is the special GM threefold as in \cite[Section 7.2]{PPZ23} and $S$ is a K3 surface, and $\pi_*:D^b(S) \rightarrow \Ku(X)$ be the forgetful functor from $\D^b(S) = \Ku(X)^{\mathbb{Z}_2}$. Also, let 
\[f_*: \sqcup_{\pi_*w = v}M_\sigma(D^b(S),w) \rightarrow M_\sigma(Ku(X),v)^{\mathbb{Z}_2}\]
be the induced surjective morphism of moduli spaces.

\begin{proposition} \label{existence_singular}
    
    For any primitive class $v\in \mathrm{K}_{\mathrm{num}}(\Ku(X))$, there exist some Gushel-Mukai threefold $X$ which makes $M_{\sigma}\bigl(\Ku(X),v\bigr)$ singular.
\end{proposition}
\begin{proof}
    With lemma \ref{lem:tau–fixed–implies–singular} and \cite[Theorem 1.8(1)]{PPZ23}, it is enough to show that there exist a K3 surface $S$ with $w\in \mathrm{K}_{num}(S)$ such that $\pi_*(w) = v$. However, this comes from \cite[Lemma 7.7, 7.8 and 7.9]{PPZ23}. Because $M_{\sigma_S}(D^b(S),w) \ne \emptyset$ as $\sigma_S$ is geometric by \cite[Lemma 7.2]{PPZ23} and the fact $w^2 \ge -2$, we are done with the proof of the lemma.
\end{proof}

\subsection{Deformation to general X}

\begin{theorem}\label{prim_irreducibility}
Let $v_0$ be a fixed primitive numerical class with $v_0^2$ odd. Then, for a general smooth Gushel--Mukai threefold $X$, the moduli space $M_\sigma(\operatorname{Ku}(X),v_0)$ is irreducible.
\end{theorem}

\begin{proof}
Proposition~\ref{existence_singular} gives a smooth GM threefold $X_0$ and a Serre-invariant stability condition $\sigma_0$ such that $M_0:=M_{\sigma_0}(\Ku(X_0),v_0)$ is singular. Theorem~\ref{thm:normality-main} or \cite[Proposition 5.1]{FGLZ25} shows that $M_0$ is normal, and Theorem~\ref{thm:singular-connectedness} then shows that $M_0$ is irreducible.

Let $\mathfrak M_3^{\mathrm{GM}}$ be the moduli stack of smooth GM threefolds. By \cite[Proposition~A.2]{kuznetsov2018derived}, it is a smooth irreducible Deligne--Mumford stack of finite type over $\mathbb C$. Choose an affine étale neighborhood $S\to\mathfrak M_3^{\mathrm{GM}}$ of $[X_0]$, with a point $t_0\in S$ over $[X_0]$, and shrink it so that $S$ is smooth and irreducible. Let $\pi\colon\mathcal X\to S$ be the induced family. By \cite[Lemma~5.9]{BP23}, this family has a relative Kuznetsov component $\Ku(\mathcal X/S)$, and after a further étale shrinking we may trivialize the relative numerical lattice and extend $v_0$ to a relative class $\mathbf v_0$.

By \cite[Corollary~26.2]{bayer2021stability}, after shrinking $S$ once more there is a stability condition $\underline\sigma=(\sigma_s)_{s\in S}$ on $\Ku(\mathcal X/S)$ over $S$, whose restriction to each fibre is one of the standard Serre-invariant stability conditions. Remark~\ref{rmk:injective-Z} and the primitivity of $\mathbf v_{0,s}$ imply that semistability and stability coincide in class $\mathbf v_{0,s}$ for every geometric point $s$. Hence \cite[Theorem~21.24(3)]{bayer2021stability} gives a proper algebraic space $p\colon\mathcal M:=M_{\underline\sigma}(\Ku(\mathcal X/S),\mathbf v_0)\to S$, whose closed fibre over $s$ is $M_{\sigma_s}(\Ku(\mathcal X_s),\mathbf v_{0,s})$. These fibres are nonempty by \cite[Theorem~1.3(1)]{PPZ23}, normal by Theorem~\ref{thm:normality-main}, and pure of dimension $d=v_0^2+1$ by \cite[Proposition~5.1(i)]{FGLZ25}. Moreover, we know that $\sigma_{t_0}$ and $\sigma_0$ lie in the same $\widetilde{\mathrm{GL}}_2^+(\mathbb R)$-orbit and therefore define the same stable objects; hence $\mathcal M_{t_0}\simeq M_0$ is irreducible.

By generic flatness, there is a dense open subset $S^{\mathrm{fl}}\subset S$ over which $p$ is flat. If $t_0\in S^{\mathrm{fl}}$, set $s_1=t_0$. Otherwise choose an integral curve $C\subset S$ through $t_0$ and meeting $S^{\mathrm{fl}}$, let $\nu\colon\widetilde C\to C$ be its normalization, and choose $c_0\in\widetilde C$ over $t_0$. The base change $\mathcal M\times_S\widetilde C\to\widetilde C$ satisfies Lemma~\ref{an_ag_lemma}, so its fibres are irreducible over an open neighborhood $V$ of $c_0$. Since $\nu^{-1}(S^{\mathrm{fl}})$ is a nonempty open subset of $\widetilde C$, we may choose a closed point $c_1\in V\cap\nu^{-1}(S^{\mathrm{fl}})$ and set $s_1=\nu(c_1)$. In either case, $s_1\in S^{\mathrm{fl}}$ and $\mathcal M_{s_1}$ is irreducible.

Now restrict $p$ to $S^{\mathrm{fl}}$. By \cite[Tag~0E0C]{stacks37.27.1}, the locus where the geometric fibre is reduced is open. It contains every closed point because every closed fibre is normal, and therefore it is all of $S^{\mathrm{fl}}$, since $S^{\mathrm{fl}}$ is of finite type over $\mathbb C$ and hence Jacobson. Thus \cite[Tag~0E1E]{stacks37.27.1} shows that the number of geometric connected components of the fibres is locally constant on $S^{\mathrm{fl}}$. Since $S^{\mathrm{fl}}$ is irreducible and $\mathcal M_{s_1}$ is irreducible, every geometric fibre over $S^{\mathrm{fl}}$ is connected. In particular, every closed fibre is connected; being also normal, it is irreducible.

Let $\mathfrak U_{v_0}\subset\mathfrak M_3^{\mathrm{GM}}$ be the image of $S^{\mathrm{fl}}$. Since $S\to\mathfrak M_3^{\mathrm{GM}}$ is étale, $\mathfrak U_{v_0}$ is a nonempty open substack, and it is dense because $\mathfrak M_3^{\mathrm{GM}}$ is irreducible. Therefore the required moduli space is irreducible for a general GM threefold and for the stability condition supplied by the family. Finally, since every Serre-invariant stability condition lies in the same $\widetilde{\mathrm{GL}}_2^+(\mathbb R)$-orbit, which does not change the stable objects or their moduli space. This proves the theorem.
\end{proof}

\begin{remark}
One possible approach to irreducibility for nonprimitive classes
$mv_0$ is to use induction in $m$, as the strategy of
\cite[Theorem~4.4]{kaledin06Sin}. By
normality, irreducibility would then follow if all irreducible
components met the strictly semistable locus. Actually, suppose that there exists a component $Y$ consisting of merely stable objects, then using the notation in Theorem \ref{thm:singular-connectedness}, one can set $G := E_0^{\oplus m}$ where $E_0 \in M_\sigma(\mathcal{C},v_0)$. Furthermore, replace $Y$ by a
projective integral variety $f: Y'\to Y$, and consider 
\[
\mathcal W[-1] \simeq \mathbf R\pi_{Y'*}\mathbf R\mathcal Hom
\bigl(\pi_X^*G,\mathcal E\bigr)
\]
as usual. By applying Mukai's trick again one can show that such a component must not exist for $M_\sigma(\mathcal{C},mv_0)$, if $(mv_0)^2$ is odd. However, the contradiction should not present again for even-square characters, so the induction fails.

\end{remark}

%
%
%



\bibliographystyle{alpha}
{\small{\bibliography{stab}}}

@article{bayer2017stability,
  title={{Stability conditions on Kuznetsov components}},
  author={Bayer, Arend and Lahoz, Mart{\'\i} and Macr{\`\i}, Emanuele and Stellari, Paolo},
  journal={(Appendix joint with Xiaolei Zhao) To appear in Ann. Sci. {\'E}c. Norm. Sup{\'e}r., arXiv:1703.10839},
  year={2017}
}

@article{JLLZ,
  title={{Categorical Torelli theorems for Gushel--Mukai threefolds}},
  author={Jacovskis, Augustinas and Lin, Xun and Liu, Zhiyu and Zhang, Shizhuo},
  journal={arXiv preprint, arXiv:2108.02946},
  year={2021}
}

@article{bridgeland,
 author = {Tom Bridgeland},
 journal = {Annals of Mathematics},
 number = {2},
 pages = {317--345},
 publisher = {Annals of Mathematics},
 title = {{Stability conditions on triangulated categories}},
 volume = {166},
 year = {2007}
}

@article{kuznetsov2018derived,
  title={{Derived categories of Gushel--Mukai varieties}},
  author={Kuznetsov, Alexander and Perry, Alexander},
  journal={Compositio Mathematica},
  volume={154},
  number={7},
  pages={1362--1406},
  year={2018},
  publisher={London Mathematical Society}
}

@article{kuznetsov2009derived,
  title={{Derived categories of Fano threefolds}},
  author={Kuznetsov, Alexander},
  journal={Proceedings of the Steklov Institute of Mathematics},
  volume={264},
  number={1},
  pages={110--122},
  year={2009},
  publisher={Springer}
}

@article{FeyzbakhshPertusi2021stab,
title={{Serre-invariant stability conditions and Ulrich bundles on cubic threefolds}},
author={Feyzbakhsh, Soheyla and Pertusi, Laura},
journal={arXiv preprint, arXiv: 2109.13549},
year={2021}
}

@article{ppzEnriques2021,
  title={{Moduli spaces of stable objects in Enriques categories}},
  author={Perry, Alexander and Pertusi, Laura and Zhao, Xiaolei},
  journal={In preparation},
  year={2021}
}

@article{bayer2021stability,
  title={Stability conditions in families},
  author={Bayer, Arend and Lahoz, Mart{\'\i} and Macr{\`\i}, Emanuele and Nuer, Howard and Perry, Alexander and Stellari, Paolo},
  journal={Publications math{\'e}matiques de l'IH{\'E}S},
  pages={1--169},
  year={2021},
  publisher={Springer}
}

@article{PPZ23,
  author    = {Perry, Alexander and Pertusi, Laura and Zhao, Xiaolei},
  title     = {MODULI SPACES OF STABLE OBJECTS IN ENRIQUES CATEGORIES},
  journal   = {arXiv:2305.10702},
  year      = {2023},
}

@article{Per21,
  author    = {Pertusi, Laura and Robinett, Ethan},
  title     = {Stability conditions on Kuznetsov components of Gushel--Mukai threefolds and Serre functor},
  journal   = {Math.\ Nachr.},
  year      = {2021},
  note      = {arXiv:2112.04769}
}

@article{Muk84,
  author    = {Mukai, Shigeru},
  title     = {Symplectic structure of the moduli space of sheaves on an abelian or K3 surface},
  journal   = {Invent. Math.},
  volume    = {77},
  pages     = {101--116},
  year      = {1984}
}

@article{kaledin06Sin,
  author    = {Kaledin, Dmitry and Lehn, Manfred and Sorger, Christoph},
  title     = {Singular symplectic moduli spaces},
  journal   = {Inventiones Mathematicae},
  volume    = {164},
  number    = {3},
  pages     = {591--614},
  year      = {2006},
  publisher = {Springer},
  doi       = {10.1007/s00222-005-0484-6},
  eprint    = {math/0504202},
  archivePrefix = {arXiv},
  primaryClass = {math.AG}
}

@article{DK18,
  author    = {Olivier Debarre and Alexander Kuznetsov},
  title     = {Gushel–Mukai varieties: classification and birationalities},
  journal   = {Algebraic Geometry},
  volume    = {5},
  year      = {2018},
  pages     = {15--76},
}

@article{BLMNSP21,
  author    = {Arend Bayer and Martí Lahoz and Emanuele Macrì and Howard Nuer and Alexander Perry and Paolo Stellari},
  title     = {Stability conditions in families},
  journal   = {Publications Mathématiques de l'IHÉS},
  year      = {2021},
  note      = {to appear, \url{https://arxiv.org/abs/2102.12514}},
}

@misc{stacks37.27.1,
  author    = {Stacks},
  title     = {Stacks Project},
  howpublished = {\url{https://stacks.math.columbia.edu}},
  note      = {Proposition 37.27.1},
  year      = {},
}

@article{ChenPertusiZhao2024,
  author       = {Chen, Huachen and Pertusi, Laura and Zhao, Xiaolei},
  title        = {Some remarks about deformation theory and formality conjecture},
  journal      = {Annali dell'Universit\`a di Ferrara},
  volume       = {70},
  number       = {3},
  pages        = {761--779},
  year         = {2024},
  doi          = {10.1007/s11565-024-00500-0}
}

@article{ArbarelloSacca2018,
  author       = {Arbarello, Enrico and Sacc\`a, Giulia},
  title        = {Singularities of moduli spaces of sheaves on K3 surfaces and Nakajima quiver varieties},
  journal      = {Advances in Mathematics},
  volume       = {329},
  pages        = {649--703},
  year         = {2018},
  doi          = {10.1016/j.aim.2018.02.016}
}

@article{Kim98,
  author  = {Kim, Hoil},
  title   = {Moduli spaces of stable vector bundles on {Enriques} surfaces},
  journal = {Nagoya Math. J.},
  volume  = {150},
  pages   = {85--94},
  year    = {1998},
  doi     = {10.1017/S002776300002506X}
}

@article{Kim06,
  author  = {Kim, Hoil},
  title   = {Stable vector bundles of rank two on {Enriques} surfaces},
  journal = {J. Korean Math. Soc.},
  volume  = {43},
  number  = {4},
  pages   = {765--782},
  year    = {2006},
  doi     = {10.4134/JKMS.2006.43.4.765}
}

@article{Hau10,
  author  = {Hauzer, Marcin},
  title   = {On moduli spaces of semistable sheaves on {Enriques} surfaces},
  journal = {Ann. Polon. Math.},
  volume  = {99},
  number  = {3},
  pages   = {305--321},
  year    = {2010},
  doi     = {10.4064/ap99-3-7}
}

@article{Nue16a,
  author  = {Nuer, Howard},
  title   = {A note on the existence of stable vector bundles on {Enriques} surfaces},
  journal = {Selecta Math. (N.S.)},
  volume  = {22},
  number  = {3},
  pages   = {1117--1156},
  year    = {2016},
  doi     = {10.1007/s00029-015-0218-6},
  eprint  = {1406.3328},
  archivePrefix = {arXiv},
  primaryClass  = {math.AG}
}

@article{Nue16b,
  author  = {Nuer, Howard},
  title   = {Projectivity and birational geometry of {Bridgeland} moduli spaces on an {Enriques} surface},
  journal = {Proc. Lond. Math. Soc. (3)},
  volume  = {113},
  number  = {3},
  pages   = {345--386},
  year    = {2016},
  doi     = {10.1112/plms/pdw033},
  eprint  = {1406.0908},
  archivePrefix = {arXiv},
  primaryClass  = {math.AG}
}

@article{Yos17,
  author  = {Yoshioka, K{\=o}ta},
  title   = {A note on stable sheaves on {Enriques} surfaces},
  journal = {Tohoku Math. J. (2)},
  volume  = {69},
  number  = {3},
  pages   = {369--382},
  year    = {2017},
  doi     = {10.2748/tmj/1505181622},
  eprint  = {1410.1794},
  archivePrefix = {arXiv},
  primaryClass  = {math.AG}
}

@article{Yos18,
  author  = {Yoshioka, K{\=o}ta},
  title   = {Moduli spaces of stable sheaves on {Enriques} surfaces},
  journal = {Kyoto J. Math.},
  volume  = {58},
  number  = {4},
  pages   = {865--914},
  year    = {2018},
  doi     = {10.1215/21562261-2017-0037},
  eprint  = {1602.06914},
  archivePrefix = {arXiv},
  primaryClass  = {math.AG}
}

@article{Bec20,
  author  = {Beckmann, Thorsten},
  title   = {Birational geometry of moduli spaces of stable objects on {Enriques} surfaces},
  journal = {Selecta Math. (N.S.)},
  volume  = {26},
  number  = {1},
  pages   = {Paper No. 14, 18 pp.},
  year    = {2020},
  doi     = {10.1007/s00029-020-0540-5},
  eprint  = {1810.02165},
  archivePrefix = {arXiv},
  primaryClass  = {math.AG}
}

@article{NY20,
  author  = {Nuer, Howard and Yoshioka, K{\=o}ta},
  title   = {{MMP} via wall-crossing for moduli spaces of stable sheaves on an {Enriques} surface},
  journal = {Adv. Math.},
  volume  = {372},
  pages   = {107283},
  year    = {2020},
  doi     = {10.1016/j.aim.2020.107283},
  eprint  = {1901.04848},
  archivePrefix = {arXiv},
  primaryClass  = {math.AG}
}

@article{Yos99,
  author  = {Yoshioka, K{\=o}ta},
  title   = {Irreducibility of moduli spaces of vector bundles on {K3} surfaces},
  journal = {arXiv preprint},
  year    = {1999},
  eprint  = {math/9907001},
  archivePrefix = {arXiv},
  primaryClass  = {math.AG}
}

@article{FGLZcubic,
  author  = {Soheyla Feyzbakhsh and Hanfei Guo and Zhiyu Liu and Shizhuo Zhang},
  title   = {{Double EPW cubes from twisted cubics on Gushel--Mukai fourfolds}},
  journal = {arXiv preprint},
  year    = {2025},
  eprint  = {2501.12964},
  archivePrefix = {arXiv},
  primaryClass  = {math.AG}
}

@article{BM14a,
  author  = {Bayer, Arend and Macr{\`i}, Emanuele},
  title   = {Projectivity and birational geometry of {Bridgeland} moduli spaces},
  journal = {J. Amer. Math. Soc.},
  volume  = {27},
  number  = {3},
  pages   = {707--752},
  year    = {2014},
  doi     = {10.1090/S0894-0347-2014-00790-6},
  eprint  = {1203.4613},
  archivePrefix = {arXiv},
  primaryClass  = {math.AG}
}

@article{BM14b,
  author  = {Bayer, Arend and Macr{\`i}, Emanuele},
  title   = {{MMP} for moduli of sheaves on {K3}s via wall-crossing: nef and movable cones, {Lagrangian} fibrations},
  journal = {Invent. Math.},
  volume  = {198},
  number  = {3},
  pages   = {505--590},
  year    = {2014},
  doi     = {10.1007/s00222-014-0501-8},
  eprint  = {1301.6968},
  archivePrefix = {arXiv},
  primaryClass  = {math.AG}
}

@article{PPZ22,
  author  = {Perry, Alexander and Pertusi, Laura and Zhao, Xiaolei},
  title   = {Stability conditions and moduli spaces for {Kuznetsov} components of {Gushel--Mukai} varieties},
  journal = {Geom. Topol.},
  volume  = {26},
  number  = {7},
  pages   = {3055--3121},
  year    = {2022},
  doi     = {10.2140/gt.2022.26.3055},
  eprint  = {1912.06935},
  archivePrefix = {arXiv},
  primaryClass  = {math.AG}
}

@article{APR22,
  author  = {Altavilla, Matteo and Petkovi{\'c}, Marin and Rota, Franco},
  title   = {Moduli spaces on the {Kuznetsov} component of {Fano} threefolds of index 2},
  journal = {{\'E}pijournal G{\'e}om. Alg{\'e}brique},
  volume  = {6},
  pages   = {Article No. 13},
  year    = {2022},
  doi     = {10.46298/epiga.2022.7047},
  eprint  = {1908.10986},
  archivePrefix = {arXiv},
  primaryClass  = {math.AG}
}

@article{LLPZ24,
  author  = {Li, Chunyi and Lin, Yinbang and Pertusi, Laura and Zhao, Xiaolei},
  title   = {Higher-dimensional moduli spaces on {Kuznetsov} components of {Fano} threefolds},
  journal = {J. Reine Angew. Math.},
  volume  = {832},
  pages   = {81--151},
  year    = {2026},
  doi     = {10.1515/crelle-2026-0001},
  eprint  = {2406.09124},
  archivePrefix = {arXiv},
  primaryClass  = {math.AG}
}

@article{Sac23,
  author  = {Sacc{\`a}, Giulia},
  title   = {Moduli spaces on {Kuznetsov} components are irreducible symplectic varieties},
  journal = {arXiv preprint},
  year    = {2023},
  eprint  = {2304.02609},
  archivePrefix = {arXiv},
  primaryClass  = {math.AG}
}

@article{AHR20,
  author  = {Alper, Jarod and Hall, Jack and Rydh, David},
  title   = {A Luna \'{e}tale slice theorem for algebraic stacks},
  journal = {Annals of Mathematics},
  volume  = {191},
  year    = {2020},
  number  = {3},
  pages   = {675--738}
}

@article{FGLZ25,
  author  = {Feyzbakhsh, Soheyla and Guo, Hanfei and
             Liu, Zhiyu and Zhang, Shizhuo},
  title   = {Lagrangian families of {Bridgeland} moduli spaces
             from {Gushel--Mukai} fourfolds},
  journal = {Compositio Mathematica},
  volume  = {161},
  number  = {8},
  pages   = {2091--2135},
  year    = {2025},
  doi     = {10.1112/S0010437X25007468}
}

@article{BP23,
  author  = {Bayer, Arend and Perry, Alexander},
  title   = {Kuznetsov's Fano threefold conjecture via
             {K3} categories and enhanced group actions},
  journal = {J. Reine Angew. Math.},
  volume  = {800},
  year    = {2023},
  pages   = {107--153},
  doi     = {10.1515/crelle-2023-0021}
}

@article{Per19,
  author  = {Pertusi, Laura},
  title   = {On the double {EPW} sextic associated to a
             {Gushel--Mukai} fourfold},
  journal = {J. Lond. Math. Soc. (2)},
  volume  = {100},
  number  = {1},
  year    = {2019},
  pages   = {83--106},
  doi     = {10.1112/jlms.12205}
}

@article{GiesekerLi96,
  author  = {Gieseker, David and Li, Jun},
  title   = {Moduli of High Rank Vector Bundles over Surfaces},
  journal = {Journal of the American Mathematical Society},
  volume  = {9},
  number  = {1},
  year    = {1996},
  pages   = {107--151},
  doi     = {10.1090/S0894-0347-96-00173-4}
}

@article{OGrady96,
  author  = {O'Grady, Kieran G.},
  title   = {Moduli of Vector Bundles on Projective Surfaces: Some Basic Results},
  journal = {Inventiones Mathematicae},
  volume  = {123},
  number  = {1},
  year    = {1996},
  pages   = {141--207},
  doi     = {10.1007/BF01232371}
}

\end{document}